\documentclass[sn-mathphys-num]{sn-jnl}

\usepackage[T1]{fontenc}
\usepackage{graphicx}
\usepackage{mathptmx}

\usepackage{enumitem}

\usepackage{epstopdf}% To incorporate .eps illustrations using PDFLaTeX, etc.
\usepackage[caption=false]{subfig}% Support for small, `sub' figures and tables

\usepackage[numbers,sort&compress]{natbib}% Citation support using natbib.sty
\bibpunct[, ]{[}{]}{,}{n}{,}{,}% Citation support using natbib.sty
\usepackage{amssymb,amsmath, amsfonts,amsthm}
\usepackage{a4wide}
\usepackage{color}

\theoremstyle{plain}% Theorem-like structures provided by amsthm.sty
\newtheorem{theorem}{Theorem}[section]
\newtheorem{lemma}[theorem]{Lemma}
\newtheorem{corollary}[theorem]{Corollary}
\newtheorem{proposition}[theorem]{Proposition}

\theoremstyle{definition}
\newtheorem{definition}[theorem]{Definition}
\newtheorem{example}[theorem]{Example}

\theoremstyle{remark}
\newtheorem{remark}{Remark}

\usepackage{url}
\usepackage{tikz}
\usepackage{pgfplots}
\pgfplotsset{compat=1.14}
\usepackage{comment}
\usepackage{cases}
\usepackage{float}
\usepackage{array}
\usepackage{tabularx}
\usepackage{ragged2e}
\usepackage{comment}
\usepackage{multirow}
\usepackage{blkarray}
\newcommand{\red}[1]{\textcolor{red}{#1}}

\newcommand{\R}{\mathbb{R}}
\newcommand{\oR}{\overline{\R}}

\newcommand{\norm}[1]{\|#1\|}

\newcommand{\dist}[1]{{\rm dist}(#1)}

\newcommand{\mv}{\,\big |\,}
\newcommand{\bmv}{\,\Big |\,}
\newcommand{\bbmv}{\,\bigg |\,}

\newcommand{\A}{{\cal A}}
\newcommand{\B}{{\cal B}}

\newcommand{\K}{{\cal K}}

\newcommand{\Sp}{{\mathcal S}}

\newcommand{\Z}{{\cal Z}}

\newcommand{\setto}[1]{\mathop{\rightarrow}\limits^#1}
\newcommand{\longsetto}[1]{\mathop{\longrightarrow}\limits^{#1}}
\newcommand{\skalp}[1]{\langle #1\rangle}

\newcommand{\xb}{\bar x}
\newcommand{\yb}{\bar y}

\newcommand{\lb}{\bar\lambda}

\newcommand{\AT}[2]{{\textstyle{#1\atop#2}}}
\newcommand{\xba}{{\bar x^\ast}}

\newcommand{\oo}{o}
\newcommand{\OO}{{\cal O}}
\newcommand{\argmin}{\mathop{\rm arg\,min}}

\newcommand{\lin}{{\rm lin\,}}

\newcommand{\cl}{{\rm cl\,}}

\newcommand{\inn}{{\rm int\,}}

\newcommand{\co}{{\rm conv\,}}
\newcommand{\gph}{\mathrm{gph}\,}
\newcommand{\cocl}{\mathrm{cocl}\,}

\newcommand{\dom}{\mathrm{dom}\,}

\newcommand{\tto}{\rightrightarrows}
\newcommand{\Limsup}{\mathop{{\rm Lim}\,{\rm sup}}}

\newcommand{\myvec}[1]{\begin{pmatrix}#1\end{pmatrix}}
\newcommand{\SCD}{SCD\ }

\newcommand{\scdreg}{{\rm scd\,reg\;}}

\newcommand{\ssstar}{semismooth$^{*}$ }

\newcommand{\ee}[2]{{#1}^{(#2)}}

\newcommand{\rge}{{\rm rge\;}}

\newcommand{\onabla}{\overline\nabla}
\newcommand{\sepa}{\,\vdots\,}

\begin{document}

\title[]{On the role of semismoothness in the implicit programming approach to selected nonsmooth optimization problems}

\author[1,2]{\fnm{Helmut} \sur{Gfrerer}}\email{helmut.gfrerer@ricam.oeaw.ac.at}

\author[1,3]{\fnm{Michal} \sur{Ko\v{c}vara}}\email{m.kocvara@bham.ac.uk}

\author[1]{\fnm{Ji\v{r}\'{i} V.} \sur{Outrata}}\email{outrata@utia.cas.cz}
\affil[1]{\orgdiv{Institute of Information Theory and Automation}, \orgname{Czech Academy of Sciences}, \orgaddress{\street{Pod vod\'{a}renskou v\v{e}\v{z}\'{\i}} 4, \city{Prague}, \postcode{18208}, \country{Czech Republic}}}

\affil[2]{\orgdiv{Johann Radon Institute for  Computational and Applied Mathematics (RICAM)}, \orgname{Austrian Academy of Sciences}, \orgaddress{\street{Altenbergerstr.\ 69}, \city{Linz}, \postcode{A-404}, \country{Austria}}}

\affil[3]{\orgdiv{School of Mathematics}, \orgname{University of Birmingham}, \orgaddress{\city{Birmingham}, \postcode{B15 2TT}, \country{U.K}}}

\newlist{myitems}{enumerate}{1}
\setlist[myitems, 1]
{label=\arabic{myitemsi}., %1., 2., 3., ...
leftmargin=2\parindent,
rightmargin=00pt%
}

\abstract{
The paper deals with the implicit programming approach to a class of Mathematical Programs with Equilibrium Constraints (MPECs) and bilevel programs in the case when the corresponding reduced problems are solved using a bundle method of nonsmooth optimization. The obtained results allow us to supply the bundle algorithm with suitable, easily computable  ``pseudosubgradients'', ensuring convergence to points satisfying a stationary condition. Both the theory and computational implementation heavily rely on the notion of SCD (subspace containing derivatives) mappings and the associated calculus. The approach is validated via a complex MPEC with equilibrium governed by a variational inequality of the 2nd kind, by a shape optimization problem with a nonsmooth objective and by an academic bilevel program with a nonsmooth upper-level objective.}

\date{}

% \begin{document}
\maketitle

\section{Introduction}
It is well-known that most rules of generalized differentiation attain the form of inclusions, unless rather restrictive qualification conditions have been applied.
%This happens in particular in case of chain %rules and, as observed, e.g., in \cite{OKZ} or %\cite{DeBa01},
This might cause difficulties in connection with computational techniques such as subgradient and bundle methods of nonsmooth optimization. Beginning with \cite{No78, No80}, mathematicians attempted to ensure convergence of these methods even in cases where exact subgradients could not be provided at every iteration. This lead to the concept of pseudosubgradients, as introduced in \cite{MiGuNo87}.
Subsequent research has recognized that this challenge can be addressed using certain semismoothness property, which does not necessarily rely on generalized gradients or generalized Jacobians; see, e.g., \cite{Kum00, ChNaQui00, HiItKu03}. A recent study, \cite{GfrOut26a}, provides an in-depth analysis of this issue in the context of the ``model'' problem
% Later it has been recognized in several works that this hurdle may be overcome on the basis of some semismoothness property, which does not necessarily rely on generalized gradients and generalized Jacobians; cf., e.g., \cite{Kum00, ChNaQui00, HiItKu03}. In the recent paper \cite{GfrOut26a}, this issue is thoroughly analyzed in connection with the ``model'' problem

\begin{align}
\nonumber  \min\ &\varphi(x,y)\\
\label{EqCompProbl}  \mbox{subject to\quad}&0\in F(x,y),\\
\nonumber &x\in U_{ad},
\end{align}
where $\varphi:\R^n\times\R^m\to\R$ is a locally Lipschitz function, $F:\R^n\times\R^m\tto\R^m$ is a set-valued mapping, and $U_{ad}\subset\R^n$ is a closed set.

The main goal of this paper is to extend the results from \cite{GfrOut26a} to the case when (\ref{EqCompProbl}) corresponds either to a {\em mathematical program with equilibrium constraints } (MPEC) or to a problem of {\em bilevel programming}.
%
% Under suitable assumptions imposed on the problem data one can, namely, convert the respective problem (\ref{EqCompProbl}) to the so-called {\em reduced}  problem in variable $x$ only and solve it, e.g., via a suitable bundle  method of nonsmooth optimization.
% One speaks about the so-called  {\em implicit programming approach} (ImP approach), which is widely used also in connection with other minimization techniques.
% Of course, this procedure has been investigated and applied already in many works including the monographs \cite{OKZ} and \cite{Dem}.
% In most of these problems, however, one is indeed not  always able to supply the used bundle method with correct Clarke subgradients of the reduced objective (as required), which is due to the above mentioned inclusions in the applied calculus rules.
% Let us mention that in  \cite{DeBa01}, in case of a special bilevel program, the authors invoked the theory of pseudosubgradients from \cite{MiGuNo87} and showed that the used bundle trust method BT \cite{SZ} converges and generates points satisfying the so-called pseudostationarity condition.
% This approach has been extended to more general bilevel programs in the monograph \cite{Dem}.

It is known that, under suitable assumptions on the problem data, one can convert problem (\ref{EqCompProbl}) to the so-called reduced problem in variable $x$ only and solve it using an appropriate method of nonsmooth optimization. This methodology is referred to as the \emph{implicit programming approach} (ImP approach) or {\em reduced approach}, and it is also widely applied in connection with other minimization techniques.
It has been extensively studied in numerous works, including the monographs \cite{LuPaRa97, OKZ, Dem}. However, in many problems it is not always possible to provide the bundle method with  correct Clarke subgradients of the reduced objective, as required. This limitation arises due to the above mentioned inclusions inherent in the applied calculus rules.

It was presumably \cite{DeBa01} where the authors for the first time  utilized the pseudosubgradient theory in a specific bilevel program of the form \eqref{EqCompProbl}. In \cite{Dem} one then finds a procedure for the computation of pseudosubgradients, provided the inclusion constraint in \eqref{EqCompProbl} amounts to KKT-conditions of a mathematical program in variable $y$, parameterized by $x$. In \cite{GfrOut26a} the authors address this issue in a general framework, by analysis of the model program \eqref{EqCompProbl}. In particular, it is shown there
\if{
Notice that, in \cite{DeBa01}, the authors addressed a specific bilevel program and, utilizing the pseudosubgradient theory developed in \cite{MiGuNo87}, they demonstrated that the bundle trust-region method (BT) \cite{SZ} converges and produces points satisfying the so-called pseudostationarity condition. This approach was extended to more general bilevel programs in the monograph \cite{Dem}.
}\fi
 that the Clarke subgradients of the reduced objective in the employed bundle method can be effectively replaced with computable pseudosubgradients, provided
\begin{enumerate}
    \item
    the mapping $S:x\mapsto y$, defined implicitly through the inclusion
$0\in F(x,y)$, satisfies some natural assumptions, and
\item
the mappings $\varphi$ and $S$ are semismooth in the prescribed sense.

\end{enumerate}
Under these conditions the described variant of the ImP approach works and generates  solutions which are stationary in a reasonable sense.

This paper deals with a particular instance of
(\ref{EqCompProbl}), when it models either an MPEC or a bilevel program. In the former case one has
%%%
$$
F(x,y)=  f(x,y) + Q(x,y),
$$
%%%
where $f:\mathbb{R}^{n} \times \mathbb{R}^{m} \rightarrow  \mathbb{R}^{m}$  and $Q:\mathbb{R}^{n} \times \mathbb{R}^{m} \rightrightarrows  \mathbb{R}^{m}$ whereas, in case of bilevel programming,
$$
S(x)=\mbox{argmin}_y\{\psi(x,y)+q(x,y)\}
$$
with $\psi$ and $q$ mapping $\mathbb{R}^{n} \times \mathbb{R}^{m}$ into $\mathbb{R}$ and $\oR$, respectively. In this way we address here a substantially broader class of lower-level equilibria than those investigated in \cite{Dem} (and also in \cite{LuPaRa97,OKZ}).

In accordance with \cite{GfrOut26a}, we rquire that multifunction $Q$  possesses the so-called SCD (subspace containing derivative) property; cf.\ \cite{GfrOut22a}. This assumption is not restrictive and allows us to weaken the assumptions needed for applying this approach, as well as to simplify the computation of the pseudosubgradients. That is why we (in Sections 3.1, 3.2) explain in detail how to compute appropriate SCD subspaces for several commonly encountered multifunctions $Q$.

As mentioned already in \cite{GfrOut26a}, the solutions of MPECs, computed by the presented approach, satisfy a stationarity condition that is always weaker (less restrictive) than the Clarke stationarity. On the other hand, under some additional assumptions, they satisfy the Clarke stationarity condition of an  equivalent MPEC, where an additional control arises as a canonical perturbation. This serves as a kind of ``seal of quality'' for this approach.

The usage of the semismoothness and the SCD properties allows us to substantially extend the class of lower-level equilibria and upper-level objectives, to which this type of ImP approach can be used. Further, we need not to evaluate elements from the limiting coderivative like in \cite{Has2}. On the other hand, since we use a standard bundle method for solving the reduced problem, our approach is limited to problems where the number $n$ of control variables is not too large.

\if{\red{As a byproduct, we may apply the concept of the so-called $\Psi$-semismoothness  also to the numerical solution of a decomposable optimization problem in two variables. This problem has the structure of the model problem (\ref{EqCompProbl}) and  can be efficiently solved by the semismooth Newton method  discussed in \cite{Kum00, Ulb11}.}
}\fi

The paper is organized as follows. The preliminary Section 2 includes basic notions of variational analysis and all elements from the theory of semismooth and SCD mappings, needed throughout the text. Section 3 focuses on MPECs. In addition to theoretical results, it contains two challenging  test examples: a complex Stackelberg game of producers/firms in an oligopolistic market adapted from \cite{GfrOutVal23} and a shape optimization problem from continuum mechanics adapted from \cite{Has2}. The approach to bilevel programs, including an illustrative test example, is discussed in Section~4. It can be viewed as an alternative to \cite[Section 6.2]{Dem}, where pseudostationary points are computed under different assumptions. In both test examples we use algorithms and software from the family BT (bundle trust), described in \cite{SZ, Zowe, SZUG}.
\if{\red{Finally, Section 5 discusses the proposed approach for handling the considered decomposable programs.}
}\fi

The following notation is employed. Given a linear subspace $L\subseteq \R^n$, $ L^\perp$ denotes its orthogonal complement.
Further, given a multifunction $F$, $\gph F:=\{(x,y)\mv y\in F(x)\}$ stands for its
graph and we denote by $\dom F$ the set of all points $x$ with $F(x)\not=\emptyset$. We call $F$ {\em single-valued}, if for every $x\in \dom F$ the set $F(x)$ is a singleton. For an element $u\in\R^n$, $\norm{u}$ denotes its Euclidean norm and  $\B_\delta(u)$ denotes the open ball around $u$ with radius $\delta$. In a product space we use the norm $\norm{(u,v)}:=\sqrt{\norm{u}^2+\norm{v}^2}$. The space of all $m\times n$ matrices is denoted by $\R^{m\times n}$. For a matrix $A\in\R^{m\times n}$ , we employ the operator norm $\norm{A}$ with respect to the Euclidean norm and denote the range of $A$ by $\rge A$. If $A$ is a square nonsingular matrix, we write $A^{-T}:=(A^{-1})^T$. We use $I_n$ to signify the $n\times n$ identity matrix. Given a set $\Omega\subset\R^s$, we write $\Omega^T:=\{y^T\mv y\in\Omega\}$ and define the distance of a point $x$ to $\Omega$ by $\dist{x,\Omega}:=\inf\{\norm{y-x}\mv y\in\Omega\}$. The respective indicator function is denoted by $\delta_\Omega$.

\section{Preliminaries}

\subsection{Nonsmooth analysis}

Consider a mapping $F:\Omega\to \R^m$, where $\Omega\subset\R^n$ is an open set. In the sequel we denote by $\OO_F$ the set of points where $F$ is (Fr\'echet) differentiable and we denote the Jacobian by $\nabla F(x)$, $x\in\OO_F$. We use this notation also for real-valued functions $f:\R^n\to\R$ when $\nabla f(x)$ is an $1\times n$ matrix, i.e., a row vector.

  The {\em B-Jacobian (B-differential)} is defined as
\[\overline \nabla F(\xb)= \Limsup_{x\longsetto{\OO_F}\xb}\{\nabla F(x)\},\ \xb\in \Omega,\]
where "Limsup" stands for the outer/upper set limit in the sense of Painleve-Kuratowski. see, e.g., \cite[Chapter 4]{RoWe98}.
Then the Clarke generalized Jacobian is given by $\co \overline\nabla F(\xb)$.

If $F$ is locally Lipschitz, by Rademacher's Theorem the set $\Omega\setminus\OO_F$ is negligible, i.e., has Lebesgue measure $0$. Since the Jacobians $\nabla F(x)$, $x\in\OO_F$, are also locally bounded by the Lipschitz constant, we obtain $\emptyset\not=\overline \nabla F(x)\subset\co\overline \nabla F(x)$, $x\in \Omega$.

\if{We say that $F$ is {\em strictly differentiable} at $\xb\in \Omega$ if there is an $m\times n$ matrix $A$ such that
\[\lim_{\AT{x,x'\to\xb}{x\not=x'}}\frac{F(x')-F(x)-A(x'-x)}{\norm{x'-x}}=0.\]
Of course, this implies that $F$ is differentiable at $\xb$ and $\nabla F(\xb)=A$.
}\fi

For a locally Lipschitz function $f:\Omega\to\R$, $\Omega\subset\R^n$ open, we can define the {\em Clarke subdifferential} as
\[\partial_c f(x) :=\co \overline \nabla f(x)^T:=\{g^T\mv g\in \co \overline \nabla f(x)\},\ x\in \Omega.\]
The elements of $\partial_c f(x)$ are called {\em subgradients} and are column vectors.

Given any extended real-valued function $f:\R^n\to\oR$, we always assume that $f$ is proper, i.e., $f(x)>-\infty$ for all $x$ and $\dom  f:=\{x\in\R^n\mv  f(x)<\infty\}\not=\emptyset$. Given a point $\xb\in \dom  f$, the {\em regular (Fr\'echet) subdifferential} of $f$ at $\xb$ is given by
\[\widehat\partial  f(\xb):=\Big\{x^*\in\R^n\mv\liminf_{x\to\xb}\frac{ f(x)- f(\xb)-\skalp{x^*,x-\xb}}{\norm{x-\xb}}\geq 0\Big\},\]
while the {\em limiting (Mordukhovich) subdifferential} is defined by
\[\partial  f(\xb):=\{x^*\mv \exists (x_k,x_k^*)\longsetto{\gph \widehat \partial f} (\xb, x^*) \mbox{ with } f(x_k)\to f(\xb)\}.\]
If $\Omega:=\dom f$ is open, then clearly $\OO_f\subset \dom\widehat\partial f$ and $\overline\nabla f(\xb)^T\subset\partial f(\xb)$, provided that $f$ is continuous at $\xb\in \Omega$. If $f$ is locally Lipschitz on $\Omega$ then
\[\partial_c f(x)=\co \partial f(x),\ x\in \Omega;\]
cf. \cite[Theorem 9.61]{RoWe98}. Further, if $f$ is a lower semicontinuous (lsc) convex function then $\partial f$  coincides with the Moreau-Rockafellar subdifferential from convex analysis.

\subsection{Variational geometry and differentiation of multifunctions}
Consider a set $A\subset\R^s$ and a point $\xb\in A$.The  {\em tangent (contingent, Bouligand) cone} to $A$ at $\bar{x}$ is given by
 \[T_{A}(\bar{x}):=\Limsup\limits_{t\downarrow 0} \frac{A-\bar{x}}{t},\]
 whereas the set
 \[\widehat{N}_{A}(\bar{x}):=(T_{A}(\bar{x}))^{\circ},\]
 i.e., the (negative) polar cone to the tangent cone $T_A(\xb)$, is the {\em regular (Fr\'{e}chet) normal cone} to~$A$ at $\bar{x}$, and
 \[N_{A}(\bar{x}):=\Limsup\limits_{\stackrel{A}{x \rightarrow \bar{x}}} \widehat{N}_{A}(x)\]
 is the {\em limiting (Mordukhovich) normal cone} to $A$ at $\bar{x}$.

\if{======================================================
If $A$ is convex, then $\widehat{N}_{A}(\bar{x})= N_{A}(\bar{x})$ amounts to the classical normal cone in
the sense of convex analysis and we will  write $N_{A}(\bar{x})$.

The limiting normal cone enables us to describe the local behavior of set-valued maps via the following generalized derivative.
===========================================================}\fi

Given a multifunction $F:\R^n\tto\R^m$ and a point $(\xb,\yb)\in \gph F$, the multifunction $D^\ast F(\xb,\yb ): \R^m \tto \R^n$,  defined by
 \[ \gph D^\ast F(\xb,\yb )=\{(y^*,x^*)\mv (x^*,-y^*)\in N_{\gph F}(\xb,\yb)\},\]
is called the {\em limiting (Mordukhovich) coderivative} of $F$ at $(\xb,\yb )$.
\if{=================================================
\begin{definition}\label{DefGenDeriv}
Consider a  multifunction $F:\R^n\tto\R^m$ and let $(\xb,\yb)\in \gph F$.

 The multifunction $D^\ast F(\xb,\yb ): \R^m \tto \R^n$,  defined by
 \[ \gph D^\ast F(\xb,\yb )=\{(y^*,x^*)\mv (x^*,-y^*)\in N_{\gph F}(\xb,\yb)\}\]
is called the {\em limiting (Mordukhovich) coderivative} of $F$ at $(\xb,\yb )$.
\end{definition}

Given a set-valued mapping $F:\R^n\tto\R^m$, we denote by $\dom F$ the set of all points $x$ with $F(x)\not=\emptyset$. We call $F$ {\em single-valued}, if for every $x\in \dom F$ the set $F(x)$ is a singleton. In this case, we can omit the second argument and write $D^*F(x)$ instead of $D^*F(x,F(x))$.

If a mapping $F:\Omega\to\R^m$, $\Omega\subset\R^n$ open, is differentiable at $\xb\in \Omega$, then it holds that $\nabla F(\xb)^Tv^*\in D^*F(\xb)(v^*),\ v^*\in\R^m$. If $F$ is even strictly differentiable at $\xb$, then one also has
\[D^*F(\xb)(v^*)=\{\nabla F(\xb)^Tv^*\},\ v^*\in\R^m.\]
========================================================}\fi
We will now recall  some basic properties of  so-called {\em SC (subspace containing) derivatives}, which were  very recently introduced in \cite{GfrOut22a,GfrOut23}.
Let us denote by $\Z_{nm}$ the metric space of all $n$-dimensional subspaces of $\mathbb{R}^{n+m}$ equipped with the metric
\begin{equation*}%\label{eq-30}
d_{\Z_{nm}}(L_{1},L_{2}):=\|P_{1}-P_{2}\|,
\end{equation*}
where $P_{i}$ is the symmetric $(n+m) \times (n+m)$ matrix representing the orthogonal projection onto $L_{i}, i=1,2$.

For notational reasons, we usually identify $\R^{n+m}$ with $\R^n\times\R^m$, i.e., for a subspace $L\in \Z_{nm}$ we  denote its elements by $(u,v)$ instead of $\myvec{u\\v}$. Analogously, given two matrices $A\in\R^{n\times n}$, $B\in\R^{m\times n}$ so that $L=\rge\!\!\begin{pmatrix}A\\B\end{pmatrix}$, we write $L=\rge(A,B)$ and we call $(A,B)$ a basis for $L$. Hence, given a matrix $C=(C_x\sepa C_y)\in\R^{l\times(n+m)}$, partitioned into submatrices $C_x\in\R^{l\times n}$ and  $C_y\in\R^{l\times m}$, the set
\[CL=C\,\rge(A,B)=C\left\{\myvec{Ap\\Bp}\bmv p\in\R^n\right\}=\{(C_xA+C_yB)p\mv p\in\R^n\}\]
is a subspace of $\R^l$. If $(A,B)\in\R^{n\times n}\times \R^{m\times n}$ is an arbitrary basis for a subspace $L\in Z_{nm}$, the set of all bases for this subspace is given by
\[\{(AC,BC)\mv C\in\R^{n\times n} \mbox{ nonsingular}\}.\]
\if{==============================================
\begin{lemma}[{\cite[Lemma 3.1]{GfrOut23}}]\label{LemBasic}
 The metric space $\Z_{nm}$ is (sequentially) compact.
\end{lemma}
=================================================}\fi

%To be consistent with the  notation in \cite{GfrOut22a} we will write $\Z_{n}$ instead of $\Z_{nn}$.

With  each $L \in \Z_{nm}$ one can associate its {\em adjoint} subspace $L^{*}$ defined by
  \begin{equation*}%\label{eq-31a}
L^{*}:=\{(-v^{*},u^{*}) \in \mathbb{R}^{m} \times \R^{n} \mv (u^{*},v^{*})\in L^{\perp}\}.
  \end{equation*}
  Since $\dim L^{\perp}=m$, it follows that $L^{*}\in \Z_{mn}$ (i.e., its dimension is $m$).
\if{================================================
It is easy to see that
\begin{equation*}%\label{eq-34}
L^{*}= S_{nm} L^{\perp}, \mbox{ where } S_{nm} =
\left( \begin{array}{lc}
0 & -I_{m}\\
I_{n} & 0
\end{array}\right).
\end{equation*}
====================================================}\fi
Further we have $(L^*)^*=L$ and $d_{\Z_{mn}}(L_1^*,L_2^*)=d_{\Z_{nm}}(L_1,L_2)$ for all subspaces $L_1,L_2\in\Z_{nm}$ and thus the mapping $L\to L^*$ is an isometry between $\Z_{nm}$ and $\Z_{mn}$; cf.\ \cite{GfrOut23}.
%%%
In what follows, the symbol $L^*$ signifies both the adjoint subspace to some $L\in \Z_{nm}$ as well as an arbitrary subspace from $\Z_{mn}$. This double role, however, cannot lead to a confusion.

\begin{definition}Consider a mapping $F:\R^n\tto\R^m$.
\begin{enumerate}
\item
We say that $F$ is {\em graphically smooth of dimension} $n$ at $(\bar{x},\bar{y})$ if $T_{\gph F}(\bar{x},\bar{y})\in \Z_{nm}$. By $\mathcal{O}_{F}$
 we denote the subset of $\gph F$, where $F$ is graphically smooth of dimension $n$.
 \item The {\em SC (subspace containing) derivative} of $F$  is the mapping  $\Sp F:\R^n \times  \R^m \tto \Z_{nm}$ given by
 \begin{align*}
\Sp F(x,y):&=\{L \in \Z_{nm} \mv \exists(x_{k},y_{k})\longsetto{\OO_F}(x,y) \mbox{ such that } \lim_{k\to\infty} d_{\Z_{nm}}\big(L, T_{\gph F}(x_{k},y_{k})\big)=0\},
\end{align*}
and the {\em adjoint SC derivative} of $F$ is the mapping  $\Sp^* F:\R^n\times\R^m \tto \Z_{mn}$  defined by
\begin{align*}
\Sp^* F(x,y):&= \{L^*\mv L\in\Sp F(x,y)\}.
\end{align*}
\end{enumerate}
\end{definition}
Both $\Sp F$ and $\Sp^*F$ are in fact special  generalized derivatives of $F$ whose elements, by virtue of the above definitions, are subspaces. Further we have the relation
\begin{gather}\label{EqSp*SubsetDF*}
L^*\subset \gph D^{*}F(x,y)\ \forall L^*\in\Sp^*F(x,y);
\end{gather}
cf.\ \cite[Eq (15)]{GfrOut23}. Let us mention that in \cite{GfrOut23} the SC adjoint derivative $\Sp^*F$ was defined in a different but equivalent manner.

Since $\Z_{nm}$ is compact, cf. \cite[Lemma 3.1]{GfrOut23},  $\dom\Sp F=\dom\Sp^*F=\cl\OO_F$. Further, $\gph \Sp F$ is closed in $\R^n\times\R^m\times \Z_{nm}$ and, due to the isometry $L \leftrightarrow L^*$, $\gph \Sp^* F$ is closed in $\R^n\times\R^m\times\Z_{mn}$ as well.
On the basis of SC derivatives we may now introduce the following notion.
\begin{definition}
A mapping
$F:\R^n\tto\R^m$ is said to have the {\em SCD property} at $(\xb,\yb)\in \gph F$, provided $\Sp^*F(\xb,\yb)\neq \emptyset$. Further $F$ has the SCD property {\em around} a point $(\xb,\yb)\in\gph F$, if $F$ has the SCD property at all points $(x,y)\in\gph F$ close to $(\xb,\yb)$. Finally, $F$ is termed an {\em SCD mapping }if it has the SCD property at all points of $\gph F$.
\end{definition}
Using the above definition, the \SCD property at $(\xb,\yb)$ is obviously equivalent with the condition $\Sp F(\xb,\yb)\not=\emptyset$.

By virtue of \eqref{EqSp*SubsetDF*},  the adjoint SC derivative $\Sp^*F$ forms some kind of skeleton for the limiting coderivative $D^*F$ for SCD mappings. It contains the information which is important in applications, in particular for numerical computations, but it is much easier to compute than the limiting coderivative. Note that in these applications we only have to compute {\em one} subspace belonging to $\Sp^*F(x,y)$ and not the whole adjoint SC derivative. On the other hand, the limiting coderivative is a more general notion important for arbitrary set-valued mappings.

\if{===========================================
The derivatives $\Sp F$ and $\Sp^*F$ can be considered as a generalization of the B-Jacobian to multifunctions. In case of single-valued continuous mappings one has the following relationship.
\begin{lemma}[{\cite[Lemma 3.5]{GfrOut23}}]\label{LemSCDSingleValued}
  Let $\Omega\subset \R^n$ be open and let $f:\Omega\to\R^m$ be continuous. Then for every $x\in \Omega$ there holds
  \begin{align*}
    %\label{EqGenBSubdiff1}
    &\Sp f(x):=\Sp\big(x,f(x)\big)\supseteq \{\rge(I_n,A)\mv A\in\overline{\nabla} f(x)\},\\
    %\label{EqGenBSubdiff2}
    &\Sp^* f(x):=\Sp^*\big(x,f(x)\big)\supseteq \{\rge(I_m,A^T)\mv A\in\overline{\nabla} f(x)\}.
  \end{align*}
  If $f$ is Lipschitz  near $x$, these inclusions hold with equality and $f$ has the \SCD property at all points $x'$ sufficiently close to  $x$.
\end{lemma}
====================================}\fi

We now provide some important examples of SCD mappings; cf.\ \cite{GfrOut22a, GfrOut23}.

\begin{example}\label{ExSCD}\begin{enumerate}~
  \item Every locally Lipschitz mapping $f:\Omega\to\R^m$, $\Omega\subset\R^n$ open, is an SCD mapping and for every $x\in \Omega$ there holds
  \begin{align*}
    %\label{EqGenBSubdiff1}
    &\Sp f(x):=\Sp\big(x,f(x)\big)= \{\rge(I_n,A)\mv A\in\overline{\nabla} f(x)\},\\
    %\label{EqGenBSubdiff2}
    &\Sp^* f(x):=\Sp^*\big(x,f(x)\big)= \{\rge(I_m,A^T)\mv A\in\overline{\nabla} f(x)\}.
  \end{align*}
  \item Every maximally monotone mapping $F:\R^n\tto\R^n$ is an SCD mapping.
  \item In particular, the subgradient mapping $\partial q$ of an lsc convex function $q:\R^n\to\oR$ is an SCD mapping. Further, $\Sp(\partial q)(x,x^*)=\Sp^*(\partial q)(x,x^*)$, $(x,x^*)\in\gph \partial q$ and every subspace $L\in \Sp(\partial q)(x,x^*)$ is self-adjoint, i.e., $L=L^*$ and admits the basis representation $L=\rge(B,I_n-B)$ for some symmetric positive semidefinite matrix $B\in\R^{n\times n}$ with $\norm{B}\leq 1$; cf.\ \cite{GfrOut22a}. For instance, when $q$ is twice continuously differentiable then $B=(I_n+\nabla^2 q(x))^{-1}$.
  \item Also the subgradient mapping of prox-regular and subdifferentially continuous functions is an SCD mapping. Recall that an lsc function $q:\R^n\to\oR$ is prox-regular at $\xb\in\dom q$ for a subgradient $\xba\in\partial q(\xb)$, if there exist $\epsilon> 0$ and $r\geq 0$ such that
\[ q(x')\geq  q(x)+\skalp{x^*,x'-x}-\frac r2 \norm{x'-x}^2\]
 whenever $x'\in \B_\epsilon(\xb)$ and $(x,x^*)\in \gph \partial  q\cap (\B_\epsilon(\xb)\times \B_\epsilon(\xba))$ with $q(x)<q(\xb)+\epsilon$. Further, $q$ is called subdifferentially continuous at $\xb$ for $\xba$ if for every sequence $(x_k,x_k^*)\longsetto{\gph\partial q}(\xb,\xba)$ one has $q(x_k)\to q(\xb)$.

 If $q:\R^n\to\oR$ is both prox-regular and subdifferentially continuous at $\xb\in\dom q$ for $\xba\in\partial q(\xb)$, then $q$ has the SCD property around $(\xb,\xba)$ and for every $(x,x^*)$ sufficiently close to $(\xb,\xba)$ there holds  $\Sp(\partial q)(x,x^*)=\Sp^*(\partial q)(x,x^*)$. Every subspace $L\in\Sp(\partial q)(x,x^*)$ is again self-adjoint and admits now a basis representation $L=\rge(P,W)$, where $P$ and $W$ are symmetric $n\times n$ matrices satisfying
 \[P^2=P,\ \ W(I_n-P)=I_n-P\,;\]
cf.\ \cite{Gfr25a}. For instance, when $q$ is twice continuously differentiable at $\xb$ then $P=I_n$, $W=\nabla^2 q(\xb)$. In case when $q$ is convex, the matrix $W$ is positive semidefinite.
\end{enumerate}
 All the examples above are special cases of {\em graphically Lipschitz mappings} (\cite[Definition 9.66]{RoWe98}), which possess also the SCD property; cf.\ \cite{GfrOut23}.
\end{example}
The following definition tailors the so-called \ssstar property introduced in \cite{GfrOut21} to SCD mappings.
\begin{definition}[{cf. \cite[Definition 5.1]{GfrOut22a}}]\label{DefSSStar}
  We say that the mapping $F:\R^n\tto\R^m$ is SCD-\ssstar  at   $(\xb,\yb)\in\gph F$ if it has the SCD property around $(\xb,\yb)$ and for every $\epsilon>0$ there is some $\delta>0$ such that
\begin{align*}
\label{EqCharSCDSemiSmoothLim}&\lefteqn{\vert \skalp{x^*,x-\xb}-\skalp{y^*,y-\yb}\vert}\\
&\nonumber\qquad\leq \epsilon
\norm{(x,y)-(\xb,\yb)}\norm{(y^*,x^*)}\ \forall(x,y)\in \gph F\cap\B_\delta(\xb,\yb)\mbox{ and } \forall (y^*,x^*)\in L^*,\ L^*\in\Sp^*F(x,y).
\end{align*}
\end{definition}

We mention here two important classes of multifunctions having this property. The next result follows from \cite[Proposition 2.10]{GfrOut22a}. \if{\red{Dort wurde die \ssstar Eigenschaft gezeigt.}}\fi
\begin{proposition}\label{PropSSstar}
  \begin{enumerate}~
    \item[(i)]Every mapping whose graph is the union of finitely many closed convex sets is SCD-\ssstar at every point of its graph.
    \item[(ii)]Every mapping with closed subanalytic graph is SCD-\ssstar at every point of its graph.
  \end{enumerate}
\end{proposition}

\subsection{Closure of mappings}
Recall that a mapping $\Psi:\R^n\tto\R^m$ is called {\em locally bounded} at $\xb\in \R^n$ if there is a neighborhood $U$ of $\xb$ and a nonnegative real $M$ such that
    \begin{equation}\label{EqBndDef}\norm{y}\leq M\ \forall y\in \Psi(x)\ \forall x\in U.\end{equation}
Further, $\Psi$ is called {\em upper semicontinuous (usc)} at $\xb$ if $\Psi(\xb)$ is closed and for every open set  $V\supset\Psi(\xb)$ the set $\{x\mv \Psi(x)\subset V\}$ is a neighborhood of $\xb$.
\if{=========================================
Recall the following definitions.
\begin{definition}
  Consider a mapping $\Psi:\R^n\tto\R^m$.
  \begin{enumerate}
    \item $\Psi$ is called {\em locally bounded} at $\xb\in \R^n$ if there is a neighborhood $U$ of $\xb$ and a nonnegative real $M$ such that
    \begin{equation}\label{EqBndDef}\norm{y}\leq M\ \forall y\in \Psi(x)\ \forall x\in U.\end{equation}
    \item $\Psi$ is called {\em outer semicontinuous (osc)} at $\xb\in\R^n$ if $\Limsup_{x\to\xb}\Psi(x)=\Psi(\xb)$.
    \item $\Psi$ is called {\em upper semicontinuous (usc)} at $\xb$ if $\Psi(\xb)$ is closed and for every open set  $V\supset\Psi(\xb)$ the set $\{x\mv \Psi(x)\subset V\}$ is a neighborhood of $\xb$.
  \end{enumerate}
  These terms are invoked {\em on $X$}, a subset of $\R^n$, when the property holds for every $x\in X$. The set $X$ will not be mentioned if $X=\R^n$.
\end{definition}
It is well-known that the Clarke subdifferential $\partial_c\vartheta$ of a Lipschitz continuous function $\vartheta:\Omega\to\R$, $\Omega\subset\R^n$ open, has the following properties:
\begin{itemize}
  \item The mapping $\partial_c \vartheta$ is locally bounded and usc  on $\Omega$.
  \item The set $\partial_c \vartheta(x)$ is nonempty, convex and compact for every $x\in \Omega$.
\end{itemize}
===================================================================}\fi

We now discuss a possibility how to extend an arbitrary locally bounded mapping to a usc mapping.
Given  $\Psi:\R^n\tto \R^m$, we denote by $\cl \Psi$ the mapping whose graph equals to the closure of $\gph\Psi$, i.e.,
\[\gph(\cl\Psi) =\cl (\gph\Psi)= \{(x,y)\mv (x,y)=\lim_{k\to\infty}(x_k,y_k)\mbox{ for some sequence }(x_k,y_k)\in\gph \Psi\}.\]
Further, we denote by $\cocl \Psi$ the mapping defined by
\[(\cocl \Psi)(x):=\co\big((\cl\Psi)(x)\big),\ x\in \R^n.\]
Obviously, $\cl(\cl\Psi)=\cl\Psi$ and for every $x\in \dom(\cl\Psi)$ the set $(\cl \Psi)(x)$ is closed. Further, $\dom(\cl \Psi)\subset\cl(\dom\Psi)$ and it is easy to see that $\dom(\cocl\Psi)=\dom(\cl\Psi)$.
\begin{lemma}[{\cite[Lemma 2.7]{GfrOut26a}}]\label{LemCocl}
  Let $\Psi:\R^n\tto\R^m$ be a  mapping. If $\Psi$ is locally bounded at $x\in\R^n$ then   $\cl\Psi$ and $\cocl \Psi$ are also locally bounded at $x$ and $(\cocl\Psi)(x)$ is a convex and compact set. Further, both $\cl\Psi$ and $\cocl \Psi$ are usc at $x$.
\end{lemma}

\begin{example}\label{ExGenJacobian}
  Let $F:\Omega\to\R^m$, $\Omega\subset\R^n$ open, be locally Lipschitz and consider the mapping $\Psi_F:\R^n\tto\R^{m\times n}$ given by
  \[\Psi_F(x):=\begin{cases}\{\nabla F(x)\}&\mbox{if $x\in \OO_F$}\\
  \emptyset&\mbox{else.}\end{cases}\]
  Then for every $x\in\Omega$ there holds $\overline \nabla F(x)=(\cl \Psi_F)(x)$ and $\co\overline \nabla F(x)=(\cocl \Psi_F)(x)$. Further $\cl\onabla F =\cl(\cl \Psi_F)=\cl \Psi_F=\onabla F$.

\end{example}

\if{=================================================
\subsection{SCD mappings}
In this section we collect only those basic facts from the theory of SCD-mappings which are needed in this paper. We refer the interested reader to \cite{GfrOut22a,GfrOut23}  for more details and more far-reaching properties.

Let us denote by $\Z_{nm}$ the metric space of all $n$-dimensional subspaces of $\mathbb{R}^{n+m}$ equipped with the metric
\begin{equation*}%\label{eq-30}
d_{\Z_{nm}}(L_{1},L_{2}):=\|P_{1}-P_{2}\|,
\end{equation*}
where $P_{i}$ is the symmetric $(n+m) \times (n+m)$ matrix representing the orthogonal projection onto $L_{i}, i=1,2$.

For notational reasons, we usually identify $\R^{n+m}$ with $\R^n\times\R^m$, i.e., for a subspace $L\in \Z_{nm}$ we  denote its elements by $(u,v)$ instead of $\myvec{u\\v}$. Analogously, given two matrices $A\in\R^{n\times n}$, $B\in\R^{m\times n}$ so that $L=\rge\!\!\begin{pmatrix}A\\B\end{pmatrix}$, we write $L=\rge(A,B)$ and we call $(A,B)$ a basis for $L$. Hence, given a matrix $C=(C_x\sepa C_y)\in\R^{l\times(n+m)}$, partitioned into submatrices $C_x\in\R^{l\times n}$ and  $C_y\in\R^{l\times m}$, the set
\[CL=C\,\rge(A,B)=C\left\{\myvec{Ap\\Bp}\bmv p\in\R^n\right\}=\{(C_xA+C_yB)p\mv p\in\R^n\}\]
is a subspace of $\R^l$. If $(A,B)\in\R^{n\times n}\times \R^{m\times n}$ is an arbitrary basis for a subspace $L\in Z_{nm}$, the set of all bases for this subspace is given by
\[\{(AC,BC)\mv C\in\R^{n\times n} \mbox{ nonsingular}\}.\]
\begin{lemma}[{\cite[Lemma 3.1]{GfrOut23}}]\label{LemBasic}
 The metric space $\Z_{nm}$ is (sequentially) compact.
\end{lemma}

%To be consistent with the  notation in \cite{GfrOut22a} we will write $\Z_{n}$ instead of $\Z_{nn}$.

With  each $L \in \Z_{nm}$ one can associate its {\em adjoint} subspace $L^{*}$ defined by
  \begin{equation*}%\label{eq-31a}
L^{*}:=\{(-v^{*},u^{*}) \in \mathbb{R}^{m} \times \R^{n} \mv (u^{*},v^{*})\in L^{\perp}\}.
  \end{equation*}
  Since $\dim L^{\perp}=m$, it follows that $L^{*}\in \Z_{mn}$ (i.e., its dimension is $m$). It is easy to see that
\begin{equation*}%\label{eq-34}
L^{*}= S_{nm} L^{\perp}, \mbox{ where } S_{nm} =
\left( \begin{array}{lc}
0 & -I_{m}\\
I_{n} & 0
\end{array}\right).
\end{equation*}
Further we have $(L^*)^*=L$ and $d_{\Z_{mn}}(L_1^*,L_2^*)=d_{\Z_{nm}}(L_1,L_2)$ for all subspaces $L_1,L_2\in\Z_{nm}$ and thus the mapping $L\to L^*$ is an isometry between $\Z_{nm}$ and $\Z_{mn}$; cf.\ \cite{GfrOut23}.
%%%
In what follows, the symbol $L^*$ signifies both the adjoint subspace to some $L\in \Z_{nm}$ as well as an arbitrary subspace from $\Z_{mn}$. This double role, however, cannot lead to a confusion.

\begin{definition}Consider a mapping $F:\R^n\tto\R^m$.
\begin{enumerate}
\item
We say that $F$ is {\em graphically smooth of dimension} $n$ at $(\bar{x},\bar{y})$ if $T_{\gph F}(\bar{x},\bar{y})\in \Z_{nm}$. By $\mathcal{O}_{F}$
 we denote the subset of $\gph F$, where $F$ is graphically smooth of dimension $n$.
 \item The {\em SC (subspace containing) derivative} of $F$  is the mapping  $\Sp F:\R^n \times  \R^m \tto \Z_{nm}$ given by
 \begin{align*}
\Sp F(x,y):&=\{L \in \Z_{nm} \mv \exists(x_{k},y_{k})\longsetto{\OO_F}(x,y) \mbox{ such that } \lim_{k\to\infty} d_{\Z_{nm}}\big(L, T_{\gph F}(x_{k},y_{k})\big)=0\},
\end{align*}
and the {\em adjoint SC derivative} of $F$ is the mapping  $\Sp^* F:\R^n\times\R^m \tto \Z_{mn}$  defined by
\begin{align*}
\Sp^* F(x,y):&= \{L^*\mv L\in\Sp F(x,y)\}.
\end{align*}
\end{enumerate}
\end{definition}
Both $\Sp F$ and $\Sp^*F$ are in fact special  generalized derivatives of $F$ whose elements, by virtue of the above definitions, are subspaces. Further we have the relation
\begin{gather}\label{EqSp*SubsetDF*}
L^*\subset \gph D^{*}F(x,y)\ \forall L^*\in\Sp^*F(x,y);
\end{gather}
cf.\ \cite[Eq (15)]{GfrOut23}. Let us mention that in \cite{GfrOut23} the SC adjoint derivative $\Sp^*F$ was defined in a different but equivalent manner.

Since $\Z_{nm}$ is compact,  $\dom\Sp F=\dom\Sp^*F=\cl\OO_F$. Further, $\gph \Sp F$ is closed in $\R^n\times\R^m\times \Z_{nm}$ and, due to the isometry $L \leftrightarrow L^*$, $\gph \Sp^* F$ is closed in $\R^n\times\R^m\times\Z_{mn}$ as well.
On the basis of these SC derivatives we may now introduce the following notion.
\begin{definition}
A mapping
$F:\R^n\tto\R^m$ is said to have the {\em SCD property} at $(\xb,\yb)\in \gph F$, provided $\Sp^*F(\xb,\yb)\neq \emptyset$. Further $F$ has the SCD property {\em around} a point $(\xb,\yb)\in\gph F$, if $F$ has the SCD property at all points $(x,y)\in\gph F$ close to $(\xb,\yb)$. Finally, $F$ is termed an {\em SCD mapping }if it has the SCD property at all points of $\gph F$.
\end{definition}
Using the above definition, the \SCD property at $(\xb,\yb)$ is obviously equivalent with the condition $\Sp F(\xb,\yb)\not=\emptyset$.

By virtue of \eqref{EqSp*SubsetDF*},  the adjoint SC derivative $\Sp^*F$ forms some kind of skeleton for the limiting coderivative $D^*F$ for SCD mappings. It contains the information which is important in applications, in particular for numerical computations, but it is much easier to compute than the limiting coderivative. Note that in these applications we only have to compute {\em one} subspace belonging to $\Sp^*F(x,y)$ and not the whole adjoint SC derivative. On the other hand, the limiting coderivative is a more general notion important for arbitrary set-valued mappings.

The derivatives $\Sp F$ and $\Sp^*F$ can be considered as a generalization of the B-Jacobian to multifunctions. In case of single-valued continuous mappings one has the following relationship.
\begin{lemma}[{\cite[Lemma 3.5]{GfrOut23}}]\label{LemSCDSingleValued}
  Let $\Omega\subset \R^n$ be open and let $f:\Omega\to\R^m$ be continuous. Then for every $x\in \Omega$ there holds
  \begin{align*}
    %\label{EqGenBSubdiff1}
    &\Sp f(x):=\Sp\big(x,f(x)\big)\supseteq \{\rge(I_n,A)\mv A\in\overline{\nabla} f(x)\},\\
    %\label{EqGenBSubdiff2}
    &\Sp^* f(x):=\Sp^*\big(x,f(x)\big)\supseteq \{\rge(I_m,A^T)\mv A\in\overline{\nabla} f(x)\}.
  \end{align*}
  If $f$ is Lipschitz  near $x$, these inclusions hold with equality and $f$ has the \SCD property at all points $x'$ sufficiently close to  $x$.
\end{lemma}

We now provide some important examples of SCD mappings; cf.\ \cite{GfrOut22a}.

\begin{example}\label{ExSCD}\begin{enumerate}~
  \item Every maximally monotone mapping $F:\R^n\tto\R^n$ is an SCD mapping.
  \item In particular, the subgradient mapping $\partial q$ of an lsc convex function $q:\R^n\to\oR$ is an SCD mapping. Further, $\Sp(\partial q)(x,x^*)=\Sp^*(\partial q)(x,x^*)$, $(x,x^*)\in\gph \partial q$ and every subspace $L\in \Sp(\partial q)(x,x^*)$ is self-adjoint, i.e., $L=L^*$ and admits the basis representation $L=\rge(B,I_n-B)$ for some symmetric positive semidefinite matrix $B\in\R^{n\times n}$ with $\norm{B}\leq 1$; cf.\ \cite{GfrOut22a}. For instance, when $q$ is twice continuously differentiable then $B=(I_n+\nabla^2 q(x))^{-1}$.
  \item Also the subgradient mapping of prox-regular and subdifferentially continuous functions is an SCD mapping. Recall that an lsc function $q:\R^n\to\oR$ is prox-regular at $\xb\in\dom q$ for a subgradient $\xba\in\partial q(\xb)$, if there exist $\epsilon> 0$ and $r\geq 0$ such that
\[ q(x')\geq  q(x)+\skalp{x^*,x'-x}-\frac r2 \norm{x'-x}^2\]
 whenever $x'\in \B_\epsilon(\xb)$ and $(x,x^*)\in \gph \partial  q\cap (\B_\epsilon(\xb)\times \B_\epsilon(\xba))$ with $q(x)<q(\xb)+\epsilon$. Further, $q$ is called subdifferentially continuous at $\xb$ for $\xba$ if for every sequence $(x_k,x_k^*)\longsetto{\gph\partial q}(\xb,\xba)$ one has $q(x_k)\to q(\xb)$.

 If $q:\R^n\to\oR$ is both prox-regular and subdifferentially continuous at $\xb\in\dom q$ for $\xba\in\partial q(\xb)$, then $q$ has the SCD property around $(\xb,\xba)$ and for every $(x,x^*)$ sufficiently close to $(\xb,\xba)$ there holds  $\Sp(\partial q)(x,x^*)=\Sp^*(\partial q)(x,x^*)$. Every subspace $L\in\Sp(\partial q)(x,x^*)$ is again self-adjoint and admits now a basis representation $L=\rge(P,W)$, where $P$ and $W$ are symmetric $n\times n$ matrices satisfying
 \[P^2=P,\ \ W(I_n-P)=I_n-P\,;\]
cf.\ \cite{Gfr24a}. For instance, when $q$ is twice continuous differentiable at $\xb$ then $P=I_n$, $W=\nabla^2 q(\xb)$. In case when $q$ is convex, the matrix $W$ is positive semidefinite.
\end{enumerate}
 All the examples above are special cases of {\em graphically Lipschitz mappings} (\cite[Definition 9.66]{RoWe98}), which possess also the SCD property; cf.\ \cite{GfrOut23}.
\end{example}
=====================================}\fi

\if{=========================================================
\subsection{Semismooth$^*$ mappings}

The following definition is adapted from \cite{GfrOut21}.

\begin{definition}\label{DefSemiSmooth}
\begin{enumerate}~
\item  A set $A\subseteq\R^n$ is called {\em \ssstar} at a point $\xb\in A$ if for every $\epsilon>0$ there is some $\delta>0$ such that
  \begin{equation*}%\label{EqCharSemiSmoothSetLim}
\vert \skalp{x^*,x-\xb}\vert\leq \epsilon \norm{x-\xb}\norm{x^*}\ \forall x\in A\cap\B_\delta(\xb)\
\forall x^*\in N_A(x).
\end{equation*}
\item
A set-valued mapping $F:\R^n\tto\R^m$ is called {\em \ssstar} at a point $(\xb,\yb)\in\gph F$, if
$\gph F$ is \ssstar at $(\xb,\yb)$, i.e., for every $\epsilon>0$ there is some $\delta>0$ such that
\begin{align}
\nonumber&\lefteqn{\vert \skalp{x^*,x-\xb}-\skalp{y^*,y-\yb}\vert}\\
\label{EqCharSemiSmoothLim}&\qquad\leq \epsilon
\norm{(x,y)-(\xb,\yb)}\norm{(y^*,x^*)}\ \forall(x,y)\in \gph F\cap\B_\delta(\xb,\yb)\ \forall
(y^*,x^*)\in\gph D^*F(x,y).
\end{align}
\end{enumerate}
\end{definition}
Note that the semismooth$^*$ property was defined in \cite{GfrOut21} in a different but equivalent way. The class of semismooth$^*$ mappings is rather broad.
We mention here two important classes of multifunctions having this property.
\begin{proposition}[{\cite[Proposition 2.10]{GfrOut22a}}]\label{PropSSstar}
  \begin{enumerate}~
    \item[(i)]Every mapping whose graph is the union of finitely many closed convex sets is \ssstar at every point of its graph.
    \item[(ii)]Every mapping with closed subanalytic graph is \ssstar at every point of its graph.
  \end{enumerate}
\end{proposition}

For single-valued functions there holds the following result.

\begin{proposition}[{\cite[Proposition 3.7]{GfrOut21}}]\label{PropNewtondiff}
Assume that $F:\R^n\to \R^m$ is a single-valued mapping
which is Lipschitz near $\xb$. Then the following two statements are equivalent.
\begin{enumerate}
  \item[(i)] $F$ is \ssstar at $\xb$.
  \item[(ii)] For every $\epsilon>0$ there is some $\delta>0$ such that
  \begin{equation*}%\label{EqCharSemiSmoothLipsch}
    \norm{F(x)-F(\xb)-C(x-\xb)}\leq\epsilon\norm{x-\xb}\ \forall x\in \B_\delta(\xb)\ \forall
    C\in\co \overline\nabla F(x),
  \end{equation*}
  i.e.,
  \begin{equation}\label{EqCharSemiSmoothLipsch1}
    \sup_{C\in \co\overline\nabla F(x)}\norm{F(x)-F(\xb)-C(x-\xb)}=\oo(\norm{x-\xb}).
  \end{equation}
\end{enumerate}
\end{proposition}
In the following definition of the SCD-\ssstar property we replace $\gph D^*F(x,y)$ in the relation \eqref{EqCharSemiSmoothLim}  by the smaller set given by the union of all subspaces $L^*\in\Sp^*F(x,y)$.

\begin{definition}\label{DefSSStar}
  We say that the mapping $F:\R^n\tto\R^m$ is SCD-\ssstar  at   $(\xb,\yb)\in\gph F$ if it has the SCD property around $(\xb,\yb)$ and for every $\epsilon>0$ there is some $\delta>0$ such that
\begin{align*}
\label{EqCharSCDSemiSmoothLim}&\lefteqn{\vert \skalp{x^*,x-\xb}-\skalp{y^*,y-\yb}\vert}\\
&\nonumber\qquad\leq \epsilon
\norm{(x,y)-(\xb,\yb)}\norm{(y^*,x^*)}\ \forall(x,y)\in \gph F\cap\B_\delta(\xb,\yb)\mbox{ and } \forall (y^*,x^*)\in L^*,\ L^*\in\Sp^*F(x,y).
\end{align*}
\end{definition}
======================================================}\fi

\subsection{On $\Psi$-semismooth single-valued mappings\label{Subsec2_6}}

The following definition is closely related to the definitions of Newton maps \cite{Kum00}, G-semismoothness \cite{Gow04}, Newton (slant) differentiability \cite{ChNaQui00, HiItKu03} and semismooth operators in Banach spaces; cf.\ \cite[Definition 3.1]{Ulb11}.
\begin{definition}[{\cite[Definition 3.1]{GfrOut26a}}]
  Consider the mapping $F:\Omega\to\R^m$, where $\Omega\subset\R^n$ is open, and the multifunction $\Psi:\R^n\tto\R^{m\times n}$.
  \begin{enumerate}
  \item We say that $F$ is {\em semismooth  with respect to $\Psi$} (shortly $\Psi$-ss) at $\xb\in\Omega$ if $\dom \Psi$ is a neighborhood of $\xb$, $\Psi$ is locally bounded at $\xb$ and
    \begin{equation*}%\label{EqGsemismooth}
      \lim_{\AT{x\setto{\Omega} \xb}{x\not=\xb}}\frac{\sup\{\norm{F(x)-F(\xb)-A(x-\xb)}\mv A\in\Psi(x)\}}{\norm{x-\xb}}=0.
    \end{equation*}
    \item We say that $F$ is  $\Psi$-ss  \emph{on a subset} $U\subset\Omega$ if $F$ is $\Psi$-ss at every $x\in U$.
    The mapping $\Psi$ is called a semismooth derivative of $F$ on $U$.
    \end{enumerate}
\end{definition}
Whereas in the context of solving nonsmooth equations $F(x)=0$ by semismooth Newton methods the semismooth property of $F$ at a solution $\xb$ is important, in this paper we are interested in the $\Psi$ -semismoothness of $F$ on an open set. We refer the interested reader to \cite{GfrOut26a} for a discussion on this subject.

Clearly, continuously differentiable mappings $F$ are $\nabla F$-ss. Moreover, $\Psi$-ss mappings obey the following chain rule.

\begin{lemma}[{\cite[Lemma 3.4]{GfrOut26a}}]\label{LemChainRule}
   Consider the  mappings $F_1:U\to\R^p$, $F_2:V\to\R^m$, where $U\subset\R^n$ and $V\subset\R^p$ are open sets with $F_1(U)\subset V$. Assume that $F_1$ is $\Psi_1$-ss  at $\xb\in U$ and that $F_2$ is $\Psi_2$-ss at $F_1(\xb)$, where $\Psi_1:\R^n\tto\R^{p\times n}$, $\Psi_2:\R^p\tto\R^{m\times p}$. Then  the composite mapping $F:U\to\R^m$ given by  $F(x):=F_2(F_1(x))$ is $\Psi$-ss  at $\xb$ with  $\Psi$ given by
  \[\Psi(x):=\{BA\mv A\in\Psi_1(x),\ B\in \Psi_2(F_1(x))\}.\]
\end{lemma}
One can also show that for given mappings $F_i:U\to\R^n$ being $\Psi_i$-ss, $i=1,2$, their sum $F_1+F_2$ is $(\Psi_1+\Psi_2)$-ss, cf. \cite[Lemma 3.5]{GfrOut26a}.

The next auxiliary statement is important in the application we have in mind.
\begin{lemma}[{\cite[Lemma 3.6]{GfrOut26a}}]\label{LemNorkin}Let $F:\Omega\to\R^m$, $\Omega\subset\R^n$ open, be continuous and assume that $F$ is $\Psi$-ss at $\xb\in\Omega$. Then $F$ is also $\cocl\Psi$-ss at $\xb$.
\end{lemma}

Consider now the following definition.
\begin{definition}[{\cite{No78,No80}}]\label{DefNorkin}
  A function $f:\R^n\to\R$ is called {\em generalized differentiable}  if there exists an usc  mapping $G:\R^n\tto\R^n$ with nonempty, convex and compact values $G(x)$ such that
  for every $\xb\in\R^n$ there holds
    \[\sup_{g\in G(x)}\vert f(x)-f(\xb)-\skalp{g,x-\xb}\vert=\oo(\norm{x-\xb}).\]
  The elements of $G(x)$ are called {\em pseudogradients} in \cite{MiGuNo87} and we will call them {\em pseudosubgradients}.
\end{definition}
By  Lemma \ref{LemNorkin}, a $\Psi$-ss function $f:\R^n\to\R$ is also generalized differentiable with
\[G(x)=\{g\in\R^n\mv g^T\in(\cocl\Psi)(x)\}.\]
Note that pseudosubgradients are column vectors whereas the elements of $\Psi$ are $1\times n$ matrices, i.e., row vectors.

 Generalized differentiable functions have many interesting properties, cf.\ \cite{No78,No80,MiGuNo87}, which can be carried over to $\Psi$-semismooth functions $f:\Omega\to\R$, $\Omega\subset\R^n$ open. Some of these results can be generalized to vector-valued mappings.
\begin{proposition}[{\cite[Proposition 3.9]{GfrOut26a}}]\label{PropNorkinR_m}
  Let $\Omega\subset\R^n$ be open and assume that $F:\Omega\to\R^m$ is $\Psi$-ss on $\Omega$. Then the following statements hold:
  \begin{enumerate}
    \item[(i)] $F$ is locally Lipschitz on $\Omega$.
    \item[(ii)] For every $x\in\Omega$ there holds $\emptyset\not=\co\overline \nabla F(x)\subset (\cocl\Psi)(x)$.
    \item[(iii)] Let $\Omega_\Psi$ denote the set of all $x\in\Omega$ such that $(\cocl\Psi)(x)$ is a singleton. Then $\Omega\setminus\Omega_\Psi$ has Lebesgue measure zero. Moreover,  the mapping $F$ is strictly differentiable at every $x\in \Omega_\Psi$ with $\{\nabla F(x)\}=(\cocl\Psi)(x)$ and the mapping $x\mapsto\nabla F(x)$ is continuous on $\Omega_\Psi$.
  \end{enumerate}
\end{proposition}
In particular, on $\Omega_\Psi$ there holds $(\cocl \Psi)=\co\onabla F(=\{\nabla F\})$ and therefore the set, where $\cocl\Psi$ differs from $\co\onabla F$, has Lebesgue measure zero. Moreover, from Proposition \ref{PropNorkinR_m}(ii) together with Lemma \ref{LemNorkin} it follows that there exists $\Psi$ such that $F$ is $\Psi$-ss on $\Omega$ if and only if $F$ is $\onabla F$-ss (or $\co\onabla F$-ss).

Consider now the situation that the mapping $F=F_l\circ F_{l-1}\circ\ldots F_1$ is the composition of various mappings $F_i:\R^{n_{i-1}}\to\R^{n_i}$, $i=1,\ldots,l$, each being semismooth with respect to $\onabla F_i$. By the chain rule (Lemma \ref{LemChainRule}) we know that $F$ is semismooth
with respect to
\[\Psi(x)=\{A_lA_{l-1}\ldots A_1\mid A_i\in\onabla F_i((F_{i-1}\circ\ldots\circ F_1)(x)),\ i=1,\ldots,l\}.\]
For numerical purposes, in many applications it suffices to compute for given $x$ the value $F(x)$ and {\em one} element $A\in\Psi(x)$.
To this end it is only necessary that for every mapping $F_i$, $i=1,\ldots,l$, we have at our disposal an oracle which returns for $x_{i-1}\in\R^{n_{i-1}}$ the value $F_i(x_{i-1})$ and one element $A_i\in\onabla F_i(x_{i-1})$. This feature is appropriate for the usage of automatic differentiation tools.
This procedure actually defines a single-valued mapping $\tilde\Psi:\R^{n_0}\to\R^{n_l\times n_0}$ with $\tilde \Psi(x)\in\Psi(x)$. Since $F$ is also $\tilde\Psi$-ss, by Proposition \ref{PropNorkinR_m}(ii) we always have that $\emptyset\not=\co\overline \nabla F(x)\subset (\cocl\tilde\Psi)(x)$, regardless which element $A_i\in\onabla F_i(x_{i-1})$ our oracle selects.

Similar considerations apply to the case when $F$ is the sum of semismooth mappings.

\subsection{On semismoothness of solution mappings to generalized equations}
In this section we consider generalized equations (GEs) of the form
\begin{equation}\label{eq-2}
0 \in F(x,y):=f(x,y)+Q(x,y),
\end{equation}
where $f:\mathbb{R}^{n} \times \mathbb{R}^{m} \rightarrow \mathbb{R}^{m}$ is continuously differentiable and $Q: \mathbb{R}^{n} \times \mathbb{R}^{m} \rightrightarrows \mathbb{R}^{m}$ has a closed graph.

For the analysis of mappings $x\mapsto y(x)$ implicitly defined by \eqref{eq-2}, the so-called {\em regular subspaces} play an important role. Let us  denote by $\Z_{m,n+m}^{\rm reg}$ the collection of all subspaces $L^*\in\Z_{m,n+m}$ satisfying
\begin{equation}\label{EqRegSubspace}(z^*,x^*,0)\in L^*\Rightarrow z^*=0,\ x^*=0.\end{equation}
By \cite[Proposition 5.6]{GfrOut26a}, for every subspace $L^*\in \Z_{m,n+m}^{\rm reg}$ there are are unique  matrices $Z_{L^*}\in\R^{m\times m}$ and $X_{L^*}\in\R^{n\times m}$ such that
$L^*=\rge(Z_{L^*}, X_{L^*},I_m)$.
It was shown in \cite[Lemma 5.9]{GfrOut26a} that for given $((x,y),z)\in\gph F$ the condition $\Sp^*F ((x,y),z)\subset \Z_{m,n+m}^{\rm reg}$ is equivalent with the boundedness of the modulus
\begin{equation*}%\label{EqSCDreg}
      {\rm scd\, reg\,}F((x,y),z):=\sup\{\norm{(z^*,x^*)}\mv (z^*,x^*,y^*)\in L^*, L^*\in \Sp^*F((x,y),z),\norm{y^*}\leq 1\}<\infty.
\end{equation*}
Using this notation, one can show the following result:
\begin{theorem}[cf. {\cite[Corollary 5.12]{GfrOut26a}}]\label{ThSCDImp}
  Consider a mapping $F:\R^n\times\R^m\tto\R^m$ and a continuous mapping $\sigma:\Omega\to\R^m$, $\Omega\subset \R^n$ open, satisfying $0\in F(x,\sigma(x))$, $x\in\Omega$. Further assume that for every $x\in\Omega$ the mapping $F$ is SCD semismooth$^*$ at $((x,\sigma(x)),0)$ and $\scdreg F((x,\sigma(x)),0)<\infty$. Then $\sigma$ is semismooth on $\Omega$ with respect to the mapping $\Psi:\R^n\tto\R^{n\times m}$ given by
  \[\Psi(x):=\begin{cases}\{-X_{L^*}^T\mv L^*\in \Sp^*F((x,\sigma(x)),0)\}&\mbox{if $x\in \Omega$},\\\emptyset&\mbox{else.}\end{cases}\]
  In particular, for every  function $\varphi:D\to \R$ defined on an open set $D\subset\R^n\times\R^m$ containing $\gph \sigma$, which is semismooth on $\gph\sigma$ with respect to some mapping $\Phi:\R^n\times\R^m\tto\R^{1\times(n+m)}$, there holds  that the function $\theta:\Omega\to \R$, $\theta(x):=\varphi(x,\sigma(x))$ is semismooth on $\Omega$ with respect to the mapping $\Theta:\R^n\tto\R^{1\times n}$ given by
  \begin{equation*}%\label{EqTheta}
  \Theta(x):=\begin{cases}
    \{g_x^T -g_y^T X_{L^*}^T\mv (g_x^T,g_y^T)\in\Phi(x,\sigma(x)), L^*\in\Sp^* F((x,\sigma(x)),0)\}&\mbox{if $x\in\Omega$},\\
    \emptyset&\mbox{else.}
  \end{cases}
  \end{equation*}
\end{theorem}
In order to apply this theorem to \eqref{eq-2}, observe that $\gph F =\{\big((x,y),z\big)\mv \big((x,y),z-f(x,y)\big)\in\gph Q\}$ and therefore
\begin{align*}%\label{EqSCDSum1}
  %\label{EqSCDSum2}
  \Sp^* F((x,y),f(x,y)+z)=\Big\{\left(\begin{matrix}I_m&0\\\nabla f(x,y)^T&I_{n+m}\end{matrix}\right)L^*\mv L^*\in \Sp^* Q((x,y),z)\Big\},
  \end{align*}
  by \cite[Theorem 4.1]{GfrOut23}
\if{==================================================================
In order to apply this theorem to \eqref{eq-2}, we need the following lemma.
\begin{lemma}\label{LemCalc}
  Assume that $Q$ has the SCD property at $((x,y),z)\in\gph Q$. Then $F$ has the SCD property at $((x,y),f(x,y)+z)$ and
  \begin{align*}%\label{EqSCDSum1}
  \Sp F((x,y),f(x,y)+z)=\Big\{\left(\begin{matrix}I_{n+m}&0\\\nabla f(x,y)&I_m\end{matrix}\right)L\mv L\in \Sp Q((x,y),z)\Big\},\\
  %\label{EqSCDSum2}
  \Sp^* F((x,y),f(x,y)+z)=\Big\{\left(\begin{matrix}I_m&0\\\nabla f(x,y)^T&I_{n+m}\end{matrix}\right)L^*\mv L^*\in \Sp^* Q((x,y),z)\Big\}.
  \end{align*}
\end{lemma}
\begin{proof}
  Follows from \cite[Proposition 2.15]{GfrOut26a}.
\end{proof}
=========================================================================}\fi
By considering basis representations $L^*=\rge(Z^*,X^*,Y^*)$ for elements of  $\Sp^*Q((x,y),z)$  with matrices $Z^*,Y^*\in \R^{m\times m}$ and $X^*\in\R^{n\times m}$, we arrive at the formula
\begin{align}\nonumber&\Sp^* F((x,y),f(x,y)+z)=\Big\{\rge(Z^*,\nabla_x f(x,y)^TZ^*+X^*,\nabla_y f(x,y)^TZ^*+Y^*)\mv \\
\label{EqBasF}&\qquad\qquad\qquad\qquad\rge(Z^*,X^*,Y^*)\in \Sp^*Q((x,y),z), Z^*,Y^*\in \R^{m\times m},\ X^*\in\R^{n\times m}\Big\}.
\end{align}
We conclude from this formula that a subspace $L^*=\rge(Z^*,\nabla_x f(x,y)^TZ^*+X^*,\nabla_y f(x,y)^TZ^*+Y^*)\in \Sp^*F((x,y),z)$ belongs to $\Z_{m,n+m}^{\rm reg}$ if and only if the $m\times m$ matrix $\nabla_y f(x,y)^TZ^*+Y^*$ is nonsingular. In this case there holds $X_{L^*}= (\nabla_x f(x,y)^TZ^*+X^*)(\nabla_y f(x,y)^TZ^*+Y^*)^{-1}$ and thus Theorem \ref{ThSCDImp} reads as follows.
\begin{corollary}\label{CorSSImplMap}
    Consider the mapping $F:\R^n\times\R^m\tto\R^m$ given by \eqref{eq-2} and a continuous mapping $\sigma:\Omega\to\R^m$, $\Omega\subset \R^n$ open, satisfying $0\in F(x,\sigma(x))$, $x\in\Omega$. Further assume that for every $x\in\Omega$ the mapping $Q$ is SCD semismooth$^*$ at $((x,\sigma(x)),-f(x,\sigma(x)))$ and for every subspace $\rge(Z^*,X^*,Y^*)\in\Sp^*Q((x,\sigma(x)),-f(x,\sigma(x)))$, where $Z^*,Y^*\in \R^{m\times m}$ and $X^*\in\R^{n\times m}$, the $m\times m$ matrix $\nabla_y f(x,\sigma(x))^TZ^*+Y^*$ is nonsingular.

    Then $\sigma$ is semismooth on $\Omega$ with respect to the mapping $\Psi:\R^n\tto\R^{m\times n}$ given by
  \[\Psi(x):=\begin{cases}\begin{array}{l}\{-(\nabla_y f(x,y)^TZ^*+Y^*)^{-T}({Z^*}^T\nabla_x f(x,y)+{X^*}^T)\mv\\
   \qquad\qquad\rge(Z^*,X^*,Y^*)\in \Sp^*Q((x,\sigma(x)),-f(x,\sigma(x)))\}\end{array}&\mbox{if $x\in \Omega$},\\\emptyset&\mbox{else.}\end{cases}\]
  In particular, for every  function $\varphi:D\to \R$ defined on an open set $D\subset\R^n\times\R^m$ containing $\gph \sigma$, which is semismooth on $\gph\sigma$ with respect to some mapping $\Phi:\R^n\times\R^m\tto\R^{1\times(n+m)}$, it holds that the function $\vartheta$, defined by  $\vartheta(x):=\varphi(x,\sigma(x))$, is semismooth on $\Omega$ with respect to the mapping $\Theta:\R^n\tto\R^{1\times n}$ given by
  \begin{equation}\label{EqTheta}
  \Theta(x):=\begin{cases}
    \{g_x^T +g_y^T A\mv (g_x^T,g_y^T)\in\Phi(x,\sigma(x)), A\in\Psi(x)\}&\mbox{if $x\in\Omega$},\\
    \emptyset&\mbox{else.}
  \end{cases}
  \end{equation}
\end{corollary}
\begin{remark}
  By Proposition \ref{PropNorkinR_m}, the semismoothness of $\sigma$ with respect to $\Psi$ ensures that $\sigma$ is locally Lipschitz on $\Omega$.
\end{remark}
Note that in case when $Q(x,y)=\tilde Q(y)$ does not depend on $x$, we have \[\Sp^*Q(x,y,z)=\{\rge (Z^*,0,Y^*)\mv \rge(Z^*,Y^*)\in\Sp^*\tilde Q(y,z)\}.\]

\section{\label{SecImP} On the ImP approach to mathematical programs with equilibrium constraints}

This section addresses the specific case of model \eqref{EqCompProbl} in which the multifunction $F$ attains the form \eqref{eq-2}.  This results in a parameter-dependent GE that is very well suited for the description of the behavior of a large class of parameter-dependent equilibrium problems. Problem
\eqref{EqCompProbl} is now formulated as the following MPEC
\begin{equation}\label{eq-1}
\begin{array}{ll}
\mbox{ minimize } & \varphi(x,y)\\
\mbox{ subject to } & 0 \in f(x,y)+Q(x,y),\\
& x \in U_{ad},
\end{array}
\end{equation}
where $x,y$ can be interpreted as the {\em control} and {\em state} variables, respectively, and $U_{ad}$ is the set of {\em admissible } controls.

% \bigskip

Throughout the whole section, we will impose the following standing assumptions:

\begin{myitems}
\item[\rm (A1)]
The {\em solution mapping} $S : \mathbb{R}^{n} \rightrightarrows \mathbb{R}^{m}$ defined by
\begin{equation}\label{eq-2.5}
S(x):= \{y \in \mathbb{R}^{m} | \mbox{ GE \eqref{eq-2} is fulfilled }\}
\end{equation}
is single-valued and continuous on an open set $\Omega \supset U_{ad}$.
\item[\rm (A2)]
$U_{ad}$ is a convex polyhedron.
\item[\rm (A3)]
$Q$ is SCD-\ssstar (in the sense of Definition \ref{DefSSStar}) at every point $((x,S(x)),-f(x,S(x)))$, $x\in\Omega$.
\end{myitems}

% \medskip

Assumption (A1) is indispensable and is imposed essentially in all works using the ImP approach to the numerical solution of MPECs. Its fulfillment can be ensured in different ways. One possibility is to combine the Mordukhovich criterion ensuring the Aubin property of $S$ with a condition guaranteeing the uniqueness of the solution to (\ref{eq-2}) for admissible controls $x \in U_{ad}$. Alternatively, if $Q$ depends only on $y$, one could employ the notion of strong regularity; cf.\ \cite[Chapter 3B]{DR}, \cite[Chapter 5]{OKZ}.
% \bigskip

Under (A1),  MPEC (\ref{eq-1}) amounts to the single-level nonsmooth optimization problem
%%%
\begin{equation}\label{eq-3}
\begin{array}{ll}
\mbox{ minimize } & \vartheta (x)\\
\mbox{ subject to } &  x \in U_{ad},
\end{array}
\end{equation}
%%%
where the objective $\vartheta : \mathbb{R}^{n} \rightarrow \mathbb{R}$ is defined via the composition $\vartheta(x):= \varphi (x,S(x)) $. Various available ImP techniques differ in the way how a solution to (\ref{eq-3}) is computed. Here we will use, similarly to \cite{OKZ}, \cite{DeBa01}, \cite{Has2}, \cite{KO}, the bundle trust-region  (BT) algorithm of Schramm and Zowe
(\cite{SZ, Zowe}) which is able to deal only with constraints in form of linear equalities and inequalities. This is the reason for imposing (A2).

% \bigskip

Assumption (A3), at the first glance, might look rather restrictive, but it turns out that it holds in case of most major classes of equilibria considered in the literature so far. In particular, it is fulfilled in all classes of MPECs considered below.

% \bigskip

In its original form, the BT algorithm requires from the user (similarly to most existing bundle methods) to compose a so-called oracle in which, given an $x \in U_{ad}$, one computes the corresponding state variable $y = S(x)$, the function value $\vartheta(x)$ and a subgradient $\xi\in \partial_c \vartheta(x)$. Then it can be shown that under the additional assumption that $\vartheta$ is weakly semismooth and under boundedness of the sequence of iterates $\ee xk$ there is an accumulation point $\hat x$ of this sequence  which is C-stationary, i.e., $0\in\partial_c\vartheta(\hat x)+N_{U_{ad}}(\hat x)$.

As observed in \cite{DeBa01}, it is not really necessary that the oracle returns a Clarke subgradient to ensure that the BT-algorithm works. Only a pseudosubgradient $\xi$ in the sense of Definition \ref{DefNorkin} is required, i.e., a certain semismoothness condition has to be fulfilled. The price we have to pay  is that no accumulation point needs  to be C-stationary.  However, it holds that at least one of them fulfills the weaker necessary optimality condition saying that $0$ is the sum of a pseudosubgradient and a normal to $U_{ad}$.

We clarify this issue in the following theorem.
\begin{theorem}[{\cite[Theorem 3.10]{GfrOut26a}}]\label{ThConvBT} Consider problem \eqref{eq-3} and assume that the objective $\vartheta$ is semismooth with respect to a mapping $\Theta:\R^n\tto\R^{1\times n}$ on an open set $\Omega\supset U_{ad}$.
Assume that the BT-algorithm is applied to this  problem along with an  oracle which returns for each $x\in U_{ad}$ besides the function value $\vartheta(x)$ an element $\xi$ with $\xi^T\in\Theta(x)$. If, moreover, the BT-algorithm  generates a bounded sequence of iterates $\ee xk$,  then this sequence possesses an accumulation point $\hat x$ satisfying the optimality condition
\begin{equation}\label{eq-9}
0 \in (\cocl \Theta)(\hat x)^T + N_{U_{ad}}(\hat x).
\end{equation}
\end{theorem}

Recall that in the ImP approach the function $\vartheta$ is given by $\vartheta(x)=\varphi(x,S(x))$. By Corollary \ref{CorSSImplMap}, $\vartheta$ is semismooth with respect to the mapping $\Theta$ given by \eqref{EqTheta} (with $\sigma$ replaced by $S$) under the following additional assumptions:
\begin{myitems}
  \item[\rm(A4)] For every $x\in\Omega$ and for every subspace $L^*=\rge(Z^*,X^*,Y^*)\in\Sp^*Q((x,S(x)),-f(x,S(x)))$, where $Z^*,Y^*\in \R^{m\times m}$ and $X^*\in\R^{n\times m}$,  the $m\times m$ matrix $\nabla_y f(x,S(x))^TZ^*+Y^*$ is nonsingular.
      \item[\rm(A5)] The function $\varphi$ is semismooth on $\{(x,S(x))\mv x\in\Omega\}$ with respect to some mapping $\Phi:\R^n\times\R^m\tto\R^{1\times(n+m)}$.
\end{myitems}
Since for every subspace $L^*\in\Sp^*Q((x,S(x)),-f(x,S(x)))$ we have $L^*\subset D^*Q((x,S(x)),-f(x,S(x)))$, and by taking into account \eqref{EqBasF}, one may conclude from the Mordukhovich criterion \cite[Theorem 9.43]{RoWe98} that (A4) is fulfilled, whenever $F$ is metrically regular around $(x,S(x),0)$, $x\in\Omega$, i.e., the implication
\[(0,0)\in D^*F(x,S(x),0)(z^*)\ \Rightarrow\ z^*=0\]
holds. This, in turn, implies that  the validity of (A4) can be ensured by assuming that the {\em enhanced} solution mapping $\widetilde{S}:\R^m\times\R^n\tto\R^m$, given by
\begin{equation}\label{EqEnhSolMap}\widetilde S(u,x):=\{y\mv u\in f(x,y)+Q(x,y)\},\end{equation}
has the Aubin property around all pairs $((0,x),S(x))$, $x\in \Omega$. Note that metric regularity of the constraints is a standard assumption (qualification condition) in optimization theory.

\if{As we will show below, Assumption (A4) is automatically fulfilled if we strengthen Assumption (A1) a little bit.
}\fi

Assumption (A5) is tailored to the situation that the mapping $\varphi$  is the composition of various semismooth mappings  as discussed at the end of Subsection \ref{Subsec2_6}.
Note that in the formulas \eqref{EqOracle1}, \eqref{eq-8} below for computing one element $\xi^T\in\Theta(x)$ we only need one element $g^T\in\Phi(x,S(x))$.

We are now in the position to explicitly formulate the oracle needed for applying the BT-algorithm to problem \eqref{eq-3}.

\medskip
\fbox{\begin{minipage}{0.95\textwidth}Oracle:
\begin{description}
\item[\rm 1)]
Given $\bar{x} \in U_{ad}$, compute $\bar{y}= S(\bar{x})$ and select some arbitrary vectors $(g_{x}^T, g_{y}^T) \in \Phi (\bar{x},\bar{y})$.
\item[\rm 2)]
Determine a subspace $L^*\in\Sp^*Q((\xb,\yb),-f(\xb,\yb))$ and its basis representation $(Z^*,X^*,Y^*)$ with matrices $Z^*,Y^*\in\R^{m\times m}$ and $X^*\in\R^{n\times m}$.
\item[\rm 3)]Compute the unique solution $\bar p$ of the linear system
\begin{equation}\label{EqOracle1}(\nabla_y f(\xb,\yb)^TZ^*+Y^*)p=-g_y\end{equation}
and set
\begin{equation}\label{EqOracle2} x^*:=(\nabla_xf(\xb,\yb)^TZ^*+X^*)\bar p.\end{equation}
\item[\rm 4)]
Return the function value $\varphi(\xb,\yb)$ and the column vector
\begin{equation}\label{eq-8}
\xi = g_x + x^{*}.
\end{equation}
\end{description}
\end{minipage}
}

\medskip
\begin{remark}
  Since $L^*=\rge(Z^*,X^*,Y^*)\subset \gph D^*Q((\xb,\yb),-f(\xb,\yb))$, we have $(X^*\bar p,Y^*\bar p)\in D^*Q((\xb,\yb),-f(\xb,\yb))(Z^*\bar p)$ and thus the inclusion
  \begin{equation}
    \myvec{x^*\\-g_y}\in\nabla f(\xb,\yb)^Tz^*+D^*Q((\xb,\yb),-f(\xb,\yb))(z^*)
  \end{equation}
  holds with $z^*=Z^*\bar p$.
\end{remark}

By Corollary \ref{CorSSImplMap}, under the posed assumptions we obtain that
$\xi$, returned by the oracle, is a pseudosubgradient of $\vartheta$ at $\bar{x}$. Thus, the convergence result of Theorem \ref{ThConvBT} holds and we want to illustrate the sharpness of condition (\ref{eq-9}). To this aim assume that a strengthened variant of (A1) is fulfilled, namely
\smallskip
\begin{myitems}
\item[\rm(A1')]
The mapping $\widetilde{S}$ given by \eqref{EqEnhSolMap}
is single-valued and locally Lipschitz on $\B_{\delta}(0) \times \Omega$ for some $\delta >0$.
\end{myitems}
\smallskip
In particular, (A1') guarantees that (A4) is fulfilled.

\begin{lemma}
  Under Assumption {\rm(A1')} there holds for every $(u,x)\in \B_\delta(0)\times \Omega$ that
  \begin{equation}\label{EqGenGrad}\Sp^*(f+Q)((x,\tilde S(u,x)),u)= \{\rge(B_u^T,-B_x^T,I_m)\mv B=(B_u\sepa B_x)\in\onabla \tilde S(u,x), B_u\in\R^{m\times m},B_x\in\R^{m\times n}\}.\end{equation}
%  Further, Assumption (A4) is automatically fulfilled.
\end{lemma}
\begin{proof}
  Consider the linear mapping $\Xi:\R^n\times\R^m\times\R^m\to\R^m\times\R^n\times \R^m$ given by $\Xi(x,y,u)=(u,x,y)$. Then
  \[\gph(f+Q)\cap (\Omega\times\R^m\times\B_\delta(0))=\{(x,y,u)\in (\Omega\times\R^m\times\B_\delta(0))\mv \Xi(x,y,u)\in\gph \tilde S\}.\]
  Since $\Sp^*\tilde S((u,x),\tilde S(u,x))=\{\rge(I_m,B^T)\mv B\in\onabla \tilde S(u,x)\}$, $(u,x)\in\B_\delta(0)\times\Omega$, equation \eqref{EqGenGrad} follows from \cite[Theorem 4.1]{GfrOut23}.
%  The last assertion is a consequence of the two basis representations \eqref{EqGenGrad} and~\eqref{EqBasF}.
\end{proof}

Further, define the function $\tilde{\vartheta}:\B_{\delta}(0) \times \Omega \rightarrow \mathbb{R} $ by

$$
\tilde{\vartheta}(u,x) = \varphi (x,\widetilde{S}(u,x))
$$
%%%
and consider the {\em enhanced} problem
%%%
\begin{equation}\label{eq-11}
\begin{array}{ll}
\mbox{ minimize } & \tilde{\vartheta}(u,x)\\
\mbox{  subject to } & (u,x) \in \{0\} \times U_{ad}.
\end{array}
\end{equation}

% \bigskip

We observe that under (A1') problem (\ref{eq-11}) is well-defined and $\hat x$ is a solution of (\ref{eq-3}) if and only if $(0,\hat x)$ is a solution of (\ref{eq-11}).

% \bigskip

\begin{proposition}
Apart from {\rm(A1')} assume that  $\varphi$ is continuously  differentiable on $\Omega\times\R^m$. Then (\ref{eq-9}) amounts to the Clarke stationarity (C-stationarity) condition for problem (\ref{eq-11}) at $(0,\hat x)$.
\end{proposition}
\begin{proof}By comparing \eqref{EqBasF} and \eqref{EqGenGrad} we can deduce that for every $x\in\Omega$ and every subspace $\rge(Z^*,X^*,Y^*)\in \Sp^*Q((x,S(x)), -f(x,S(x)))$, $Z^*,Y^*\in\R^{m\times m}$, $X^*\in\R^{n\times m}$, there is a matrix $B=(B_u\sepa B_x)\in\onabla\tilde S(0,x)$ such that
\[B_u^T= Z^*(\nabla_y f(x,y)^TZ^*+Y^*)^{-1},\quad -B_x^T=(\nabla_x f(x,y)^TZ^*+X^*)(\nabla_y f(x,y)^TZ^*+Y^*)^{-1},\]
and vice versa. Hence the mapping $\Theta$ from Corollary \ref{CorSSImplMap} has the representation
\[\Theta(x)=\{\nabla_x\varphi(x,S(x))+ \nabla_y\varphi(x,S(x))B_x\mv B=(B_u\sepa B_x)\in\onabla \tilde S(0,x)\},\]
implying $\Theta(x)=(\cl\Theta)(x)$ and
\[(\cocl\Theta)(x)=\{\nabla_x\varphi(x,S(x))+ \nabla_y\varphi(x,S(x))B_x\mv B=(B_u\sepa B_x)\in\co\onabla \tilde S(0,x)\}.\]
 By the chain rule \cite[Theorem 2.6.6]{Cla83} we have
\[\partial_c \tilde\vartheta(0,x)^T=\{(\nabla_y\varphi(x,S(x))B_u, \nabla_x\varphi(x,S(x))+\nabla_y\varphi(x,S(x))B_x)\mv B=(B_u\sepa B_x)\in\co\onabla \tilde S(0,x)\}.\]
Thus, when the optimality condition \eqref{eq-9} holds, we have
\[\begin{pmatrix}u^*\\0\end{pmatrix}\in  \partial_c \tilde\vartheta(0,\hat x)+\{0\}\times N_{U_{ad}}(\hat x)\]
for some $u^*\in\R^m$. Since $N_{\{0\}\times U_{ad}}(0,\hat x)=\R^m\times N_{U_{ad}}(\hat x)$, we can conclude that the C-stationarity condition
\[0\in \partial_c \tilde\vartheta(0,\hat x)+N_{\{0\}\times U_{ad}}(0,\hat x)\]
for the enhanced problem is fulfilled.
\end{proof}

Although the multifunction $(\cocl \Theta)^T$ differs from $\partial_c\vartheta$ on $\Omega$ only on a set of measure $0$, we are aware that the optimality condition (\ref{eq-9}) can be definitely less sharp than standard C-stationarity conditions for (\ref{eq-3}). This can be explained by the theory of explicit and implicit variables (\cite{BM}) and it is documented by the following example.

\begin{example}
Consider an MPEC (\ref{eq-1}) with $x \in \mathbb{R}, y \in \mathbb{R}^{2}, U_{ad} = [-1, 1]$,
%%%
$$
\varphi(x,y)= -0.5 y_{1}+y_{2},
$$
%%%
and an equilibrium governed by the linear complementarity problem
%%%

$$
0 \in f(x,y)+Q(y):=\left( \begin{array}{l}
y_{1}+y_{2}-x\\
y_{2}+x
\end{array}\right) + N_{\R^{2}_{+}}(y).
$$
%%%
It is easy to see that
\begin{equation}\label{eq-24}
S(x)= \left\{ \begin{array}{ll}
(x,0)^{T} & \mbox{ for } x \geq 0,\\
(0,-x)^{T} & \mbox{ otherwise, }
\end{array}  \right.
\qquad
\vartheta(x)=\begin{cases}
  -0.5x&\mbox{if $x\geq0$,}\\
  -x&\mbox{otherwise,}
\end{cases}
\end{equation}
%%%
and so assumption (A1) is fulfilled. Clearly, (A2) is fulfilled as well and (A3) holds true by \cite[Proposition 5.3]{GfrOut23}.

We observe that
\[N_{\R^2_+}(S(x))\ni-f(x,S(x))=\begin{cases}
  (0,-x)^T&\mbox{if $x\geq0$,}\\
  (2x,0)^T&\mbox{else}
\end{cases}\]
and
\begin{align*}&\Sp^* N_{\R^2_+}(S(x),-f(x,S(x)))\\
&=\begin{cases}\{(\R\times\{0\})\times(\{0\}\times\R)\}&\mbox{if $x>0$,}\\
\{\R^2\times\{(0,0)\},(\R\times\{0\})\times(\{0\}\times\R), (\{0\}\times\R)\times(\R\times\{0\}),\{(0,0)\}\times\R^2\}&\mbox{if $x=0$,}\\
\{(\{0\}\times\R)\times(\R\times\{0\})\}&\mbox{if $x<0$\,;}
\end{cases}
\end{align*}
cf.\ Section \ref{SubsecPolyhedral} below. Some calculations yield that the mappings $\Psi$ and $\Theta$ from Corollary \ref{CorSSImplMap} are given by
\begin{gather*}
\Psi(x)=\begin{cases}\left\{\myvec{1\\0}\right\}&\mbox{if $x>0$,}\\[2.5ex]
\left\{\myvec{2\\-1},\myvec{1\\0},\myvec{0\\-1}, \myvec{0\\0}\right\}&\mbox{if $x=0$,}\\[2.5ex]
\left\{\myvec{0\\-1}\right\}&\mbox{if $x<0$,}\end{cases}\quad
\Theta(x)=\begin{cases}\{-0.5\}&\mbox{if $x>0$,}\\
\{-2,\ -0.5,\ -1,\ 0\}&\mbox{if $x=0$,}\\
\{-1\}&\mbox{if $x<0$.}
\end{cases}
\end{gather*}
We see that for all $x\not=0$ there holds $\Psi(x)=\onabla S(x)$, $\Theta(x)=\partial_c\vartheta(x)$, but $\Psi(0)\not=\onabla S(0)$ and $\Theta(0)\not\subset\partial_c\vartheta(0)=[-1,-0.5]$.

Consequently, if the BT-algorithm encounters one iterate $\ee xk=0$, our oracle could return the pseudosubgradient $\xi=0$ and the BT-algorithm would stop. Nevertheless, note that if we start with a randomly chosen starting point $\ee x0$, then all iterates $\ee xk$ will be different from $0$ with probability 1 and actually, the BT-algorithm will terminate after a few iterations with the exact solution  $\hat x=1$.
\end{example}

The only complicated step in the above approach is the specification of a suitable linear subspace $L^*\in\Sp^*Q(\xb,\yb)$ and its  basis representation as range of suitable matrices $Z^*,X^*,Y^*$ (or $Z^*,Y^*$, if $Q$ does not depend on $x$).

In the next sections we will explain how these subspaces and their basis representations can be computed in some frequently arising cases.
% \bigskip
\subsection{\label{SubsecPolyhedral}Normal cone mapping to convex polyhedral sets}
Consider first the (easy) case that $Q(y)=\widehat N_C(y)$, where
\[C:=\{y\in\R^m\mv \skalp{a_i,y}\leq b_i, i=1,\ldots,l \}\]
is a convex polyhedral set having the representation above with vectors $a_i\in\R^m$ and reals $b_i$, $i=1,\ldots,l$. Given $(\yb,\yb^*)\in \gph N_C$ it is well-known (\cite[Proposition 5.3]{GfrOut23}) that
\[\Sp N_C(\yb,\yb^*)=\Sp^* N_C(\yb,\yb^*)=\{(F-F)\times(F-F)^\perp\mv F \mbox{ is face of }\K_C(\yb,\yb^*)\},\]
where $\K_C(\yb,\yb^*)=\{v\in T_C(\yb)\mv \skalp{\yb^*,v}=0\}$ denotes the critical cone to $C$ at $\yb$ with respect to $\yb^*$.
The computation of {\em all} the subspaces in $\Sp^* N_C(\yb,\yb^*)$ might be  puzzling, but for our oracle we only need  {\em one} of them. This subspace can be chosen as the one which is induced by  the face $E:=\lin T_C(\yb)$, the {\em lineality space}  of the tangent cone $T_C(\yb)$ given by $\lin T_C(\yb)=\{v\mv\skalp{a_i,v}=0,\ i\in I(\yb)\}$, where
\[I(\yb):=\{i\in\{1,\ldots,l\}\mv \skalp{a_i,\yb}=b_i\}.\]
Note that this lineality space $E$ is for every normal $y^*\in N_C(\yb)$ a face of the critical cone $\K_C(\yb,y^*)$.
 Since $E$ is a subspace, we have $E=E-E$ and the subspace $L^*$, required for our oracle, can be chosen as $L^*=E\times E^\perp$. There remains to compute a basis representation of $L^*$.
 To this aim we may employ the (full) QR factorization of the matrix $D$, where $D$ is composed from the normalized vectors $a_i/\norm{a_i}$, $i\in I(\yb)$, as columns. Let
\[DP=Q\begin{pmatrix}R_{11}&R_{12}\\0&0\end{pmatrix},\]
where $P$ is an $\vert I(\yb)\vert\times \vert I(\yb)\vert$ permutation matrix, $Q\in\R^{m\times m}$ is orthogonal, $R_{11}\in\R^{s\times s}$ is a nonsingular upper triangular matrix,  $R_{12}\in\R^{s\times(\vert I(\yb)\vert-s)}$ and $s$ is the rank of $D$. If we partition $Q=(Q_1\sepa Q_2)$ with $Q_1\in\R^{m\times s}$, $Q_2\in\R^{m\times(m-s)}$, there holds $E=\rge Q_2$, $E^\perp=\rge Q_1$ and consequently $L^*=E\times E^\perp=\rge(Z^*,Y^*)$ with $Z^*=(Q_2\sepa 0)$, $Y^*=(0\sepa Q_1))$. Here we use the convention that the range of an $m\times 0$ matrix is $\{0\}\subset\R^m$.
\subsection{\label{SubsecInequ} Normal cone mapping to inequality   constraints}

Consider now the case when $Q(y)=\widehat{N}_{\Gamma}(y)$ with
\begin{equation}\label{eq-200}
\Gamma = \{y \in \mathbb{R}^{m} | g_{i}(y)\leq 0, ~ i=1,2,\ldots, l\},
\end{equation}
%%%
where the functions $g_{i}: \mathbb{R}^{m} \rightarrow  \mathbb{R}$ are convex  and twice continuously differentiable. This situation has been thoroughly investigated in \cite{OKZ} (even for $x$-dependent inequalities and equalities) under the {\em  linear independence constraint qualification} (LICQ). Here we present a modification of this approach which is simpler and enables us to replace LICQ by the following (weaker) assumption.
\begin{myitems}
\item [\rm(CQ):]
For all $x \in U_{ad}$ and $y = S(x)$ the inequality system in (\ref{eq-200}) fulfills the {\em Mangasarian-Fromovitz} CQ (denoted by MFCQ) and the {\em constant rank} CQ (denoted by CRCQ).
\end{myitems}
(CQ) will allow us to invoke a result from \cite{Hen1} or \cite{GfrOut16}, essential in the proof of Proposition~\ref{PropTangIneq} below.

So, throughout the rest of this section we assume that (CQ) is fulfilled and consider a triple $(\bar{y},\bar{y}^*,\bar{\lambda})$, where $\bar{y}=S(\bar{x})$ for some admissible control $\bar{x},\yb^*\in Q(\bar{y})$ and $\bar{\lambda}\in \mathbb{R}^{l}$ is an arbitrary {\em Lagrange multiplier} satisfying the conditions
%%%
\begin{equation}\label{eq-101}
\yb^*=\sum\limits^{l}_{i=1}\bar{\lambda}_{i}\nabla g_{i}(\bar{y})^T, ~~\bar{\lambda}\geq 0, ~~
\langle \bar{\lambda}, g(\bar{y})\rangle =0
\end{equation}
%%%
with $g(\bar{y})=(g_{1}(\bar{y}), \ldots, g_{l}(\bar{y}))^T$.

Further, for the sake of notational simplicity, we introduce the notation
$$
A:=\sum\limits^{l}_{i=1}\bar{\lambda}_{i}\nabla^{2} g_{i}(\bar{y})=A^{T} ~ \mbox{ and } ~
\K_{\Gamma}(\bar{y},\yb^*):=\{h\in T_\Gamma(\yb)\mv\skalp{\yb^*,h}=0\},
$$
i.e., it is the critical cone to $\Gamma$ at $\bar{y}$ with respect to $\yb^*$. Let
$$
I(y):= \{i\in\{1,2,...,l\}| g_i(y)=0\}
$$
denote the index set of active inequalities (at $y$).
Recall that the critical cone $\K_\Gamma(\yb,\yb^*)$ has the representation
\[\K_\Gamma(\yb,\yb^*)=\left\{z\bbmv \begin{matrix}\nabla g_i(\yb)z=0,\ i\in I(\yb):\ \lb_i>0\\
\nabla g_i(\yb)z\leq0,\ i\in I(\yb):\ \lb_i=0\end{matrix}\right\}.\]

\begin{proposition}\label{PropTangIneq}
In the setting described above let
$$
E:=\{z \in \mathbb{R}^{m} | \nabla g_{i}(\bar{y}) z  = 0 ~ \mbox{ for } ~ i \in I(\bar{y})\},
$$
and consider the subspace $M \in \Z_{mm}$ given by
$$
M = E \times E^{\perp}.
$$
%%%
Then for
$$
L: = \left \{ (h,k)\in \mathbb{R}^{m}  \times \mathbb{R}^{m} \left | \left (\begin{array}{l}
h\\
k-Ah
\end{array}\right ) \in M  \right.  \right  \}
$$
one has
%%%
\begin{equation}\label{eq-201}
L=L^* \in \mathcal{S}Q(\bar{y},\yb^*)= \mathcal{S}^{*}Q(\bar{y},\yb^*).
\end{equation}
\end{proposition}

% \medskip

\proof
Consider a multiplier $\bar{\lambda}$ satisfying (\ref{eq-101}) and let $J:=\{i\in I(\yb)\mv \lb_i=0\}$  denote the index set of active constraints with vanishing multipliers.
Next let us construct a perturbed multiplier $\tilde{\lambda}$ such that, with some $\varepsilon > 0$,
$$
\begin{array}{lll}
\tilde{\lambda}_{i} = \bar{\lambda}_{i} & \mbox{ for } & i \in I(\bar{y}) \setminus J\\
\tilde{\lambda}_{i} = \varepsilon & \mbox{ for } & i \in J.
\end{array}
$$
It follows that with
$$
y^*_{\varepsilon}:= \sum\limits_{i \in I(\bar{y})}\tilde{\lambda}_{i}\nabla g_{i}(\bar{y})
$$
one has
$$
\K_{\Gamma}(\bar{y},y^*_{\varepsilon})=E.
$$
We may now invoke \cite[Theorem 3.1]{Hen1} or \cite[Theorem 1]{GfrOut16}, according to which
%%%
$$
T_{\gph Q}(\bar{y},y^*_{\varepsilon})= \{(h,k)\in \mathbb{R}^{m}  \times \mathbb{R}^{m} | k \in A_{\varepsilon}h + N_{E}(h)\} =
\{(h,k)| h \in E, k-A_{\varepsilon}h \in E^{\perp}\},
$$
where $A_{\varepsilon} = \displaystyle\sum\limits_{i \in I(\bar{y})} \tilde{\lambda}_{i}\nabla^{2}q_{i}(\bar{y})$. Since $T_{\gph Q}(\bar{y},y^*_{\varepsilon})$ admits the representation
\begin{equation}\label{eq-203}
T_{\gph Q}(\bar{y}, y^*_{\varepsilon}) =
\left( \begin{array}{ll}
I & 0\\
-A_{\varepsilon} & I
\end{array}  \right )^{-1}  M=\left( \begin{array}{ll}
I & 0\\
A_{\varepsilon} & I
\end{array}  \right )  M
\end{equation}
and $M\in\Z_{mm}$, it follows that $T_{\gph Q}(\bar{y},y^*_{\varepsilon})\in \Z_{mm}$ as well.

To prove (\ref{eq-201}) it suffices now to invoke \cite[Lemma 3.1(iii)]{GfrOut22a}. Since $A_{\varepsilon}\rightarrow A$ for $\varepsilon \searrow 0$, it follows from \cite[Lemma 3.1]{GfrOut22a} that
\begin{equation}\label{EqReprW}
\lim\limits_{\varepsilon \searrow 0}
\left( \begin{array}{ll}
I & 0\\
A_{\varepsilon} & I
\end{array}  \right )M=\begin{pmatrix}I&0\\A&I\end{pmatrix}M=L \in \mathcal{S}Q(\bar{y},\yb^*).
\end{equation}
Finally, since $Q$ is the subdifferential mapping of the convex lsc function $\delta_\Gamma$, the relations $L=L^*$ and $\Sp Q(\bar{y},\yb^*) =
\Sp^*Q(\bar{y},\yb^*)$ follow and we are done.
\endproof

It remains to find a suitable basis representation of $L^*$. Denoting by $D$ the $m\times \vert I(\bar y)\vert$ matrix having the normalized vectors $\nabla g_i(\yb)^T/\norm{\nabla g_i(\yb)}$, $i\in I(\yb)$, as columns, we can proceed as in Section \ref{SubsecPolyhedral} to find an orthogonal $m\times m$ matrix $Q=(Q_1\sepa Q_2)$  with $Q_1\in\R^{m\times s}$, $Q_2\in\R^{m\times(m-s)}$, where $s$ denotes the rank of $D$, so that $M=\rge((Q_2\sepa 0),(0\sepa Q_1))$ and  by using \eqref{EqReprW} we obtain that
\[L^*=\rge (Z^*,Y^*)\mbox{ with }Z^*=(Q_2\sepa 0),\ Y^*=(AQ_2\sepa Q_1).\]

Thus the linear system \eqref{EqOracle1} and  formula \eqref{EqOracle2} amount to
\begin{equation}\label{eq-206}
(\nabla_{y}f(\bar{x},\bar{y})^{T}+A)Q_2p_1+ Q_1p_2=-g_y\quad\mbox{with}\quad p_1\in\R^{m-s},\ p_2\in\R^s
\end{equation}
and
\begin{equation}
\begin{array}{ll}
& x^{*}=\nabla_{x}f(\bar{x},\bar{y})^{T}Q_{2}\bar p_1.
\end{array}
\end{equation}
By premultiplying the linear system \eqref{eq-206} by $Q_2^T$, the computation of $\bar p_1$ simplifies to the solution of the reduced system
\[Q_2^T(\nabla_{y}f(\bar{x},\bar{y})^{T}+A)Q_2p_1=-Q_2^Tg_y\]
in $m-s$ variables. This system is smaller than the one used in the approach  from \cite{OKZ} where a  system of $m+s$ equations has to be solved. On the other hand, the approach of \cite{OKZ} enables us to handle the  situation, when $\varphi$ depends also on the multiplier $\lambda$.
\if{
\begin{remark}
  Note that under Assumption (CQ), we  only know {\em one} subspace $L^*\in \Sp^*Q(\yb,\yb^*)$. It might be subject of further research to compute the whole SC derivative   $\Sp^*Q(\yb,\yb^*)$.
\end{remark}
}\fi

\subsection{\label{SecStackel}Stackelberg games}
Consider the case when the  equilibrium is governed by a variational inequality of the second kind that can be written down in the form
\eqref{eq-2} with
\begin{equation}\label{eq-60}x\in\R^n,\ y=(y^1,y^2,\ldots,y^l)\in(\R^n)^l\quad\mbox{and}\quad Q(x,y)=Q(y)=\partial q(y)=\prod_{i=1}^l\partial q^i(y^i).
\end{equation}
In \eqref{eq-60} functions $q^i:\R^n\to\oR$ are proper convex and lower semicontinuous.

In this way one can model, e.g., a parameter-dependent Nash game, where $x$ is the parameter, $y^i$ is the strategy of the $i$-th player, $i=1,2,\ldots,l$, and $q^i$ is the nonsmooth part of his objective which does not depend on the strategies of the remaining players. Put $m:=nl$.  As already pointed out in Example \ref{ExSCD},  $Q=\partial q$ is an SCD mapping. Further, due to the imposed separability of $q$ and  \cite[Proposition 3.5]{GfrMaOutVal23}, the matrices $Z^*,Y^*$
 have the block diagonal structure
\begin{equation}\label{EqStackelbergSep}
Z^*=\begin{pmatrix} Z_1^*&0&\cdots&0\\
0&Z_2^*&&0\\
\vdots&&\ddots&\vdots\\
0&0&\cdots&Z_l^*\end{pmatrix}
\quad Y^*=\begin{pmatrix} Y_1^*&0&\cdots&0\\
0&Y_2^*&&0\\
\vdots&&\ddots&\vdots\\
0&0&\cdots&Y_l^*\end{pmatrix},
\end{equation}
where $\rge(Z_i^*,Y_i^*)\in\Sp^*\partial q^i(\bar y^i,{\yb^*}{}^i)$, $i=1,2,\ldots,l$, with $\yb^*=-f(\xb,\yb)$.

As a numerical test example of an MPEC with such an  equilibrium we will consider an oligopolistic market, where each player (firm) produces at most $n$ commodities and tries to maximize his profit. When each player knows all data of his rivals, the equilibrium  state of such a market is the Cournot--Nash non-cooperative equilibrium (\cite[Chapter 12]{OKZ}). Consider now the situation, when one of the players decides to replace his original non-cooperative strategy by the Stackelberg one  and becomes thus the Leader. The remaining players share the rest of the market by computing a new Cournot--Nash equilibrium among themselves. In this case the Leader has to solve MPEC \eqref{eq-1}, where $x=(x_1,x_2,\ldots,x_n)\in\R^n$ is his own production, $U_{ad}(=:\Omega^0)$ is his set of admissible productions and $\varphi:\R^n\times\R^m\to\R$ is his loss function. In the respective GE \eqref{eq-2}  the $i$-th component of the single-valued part $f:\R^n\times\R^m\to\R^m$ amounts to
\[f^i(x,y)=\nabla_{y^i}\psi^i(x,y)\in\R^n,\ i=1,2,\ldots,l,\]
where $\psi^i:\R^m\times\R^m\to\R$ is the loss function of the $i$-th firm (depending naturally on all strategies of his rivals including the Leader). Finally, the multi-valued part is given by \eqref{eq-60} and
\begin{equation}\label{eq-61}q^i(y^i)=\tilde q^i(y^i)+\delta_{\Omega^i}(y^i),\ i=1,\ldots,l,
\end{equation}
where $\Omega^i\subset\R^n$ is the closed convex set of admissible strategies  of the $i$-th player.
The convex nonsmooth functions $\tilde q^i;\R^n\to\R$  model the so-called costs of change, representing the expenses associated with any change of the production profile of the $i$-th firm. Naturally, these expenses may be nil. In particular, we choose $\tilde q^i(y^i)=\sum_{j=1}^n\beta_j^i\vert y^i_j-a^i_j\vert$, where $a^i$ denotes the  ''previous'' production portfolio of the $i$-th player and $\beta^i_j$ are positive weights. The sets $\Omega^i$ are convex polyhedral sets of the form $\{y^i\mv \xi^i_j y^i\leq\zeta^i_j, j=1,\ldots,p^i\}$ and the loss functions are
\[\varphi(x,y)=c^0(x)-\sum_{j=1}^nx_j\pi_j(t_j(x,y))+\tilde q^0(x),\ \psi^i(x,y):=c^i(y^i)-\sum_{j=1}^ny^i_j\pi_j(t_j(x,y)),\]
where $t(x,y):=x+\sum_{i=1}^ly^i\in\R^n$ is the sum of all available commodities on the market and $\pi_j:\R_+\to\R_+$ denotes the so-called {\em inverse demand function} for the $j$-th commodity.
The concrete form of the continuously differentiable functions $c^i$ and $\pi_j$ has been taken over from \cite{GfrOutVal23}. It is important to note that the assumptions, imposed in \cite{GfrOutVal23} on the above listed problem data, ensure in particular the fulfillment of principal assumptions (A1)-(A5) and thus the possibility to solve this MPEC via the suggested approach.  Concretely,
by \eqref{EqStackelbergSep} we   compute subspaces
$L^i\in \Sp^*(\partial q^i)(\yb^i,\yb^*{}^i)$, $i=1,\ldots,l$, and by \cite[Theorem 7.1]{GfrOutVal23} we have
$L^i=E^i\times {E^i}^\perp$ with
\[E^i=\{w\in\R^n\mv w_j=0, j\in J^i_0(\yb^i),\ \skalp{\xi^i_j,w}=0, j\in I^i_0(\yb^i)\}\]
and $J^i_0(\yb^i)=\{j\mv \yb^i_j=a^i_j\},\ I^i_0(\yb^i)=\{j\mv \skalp{\xi^i_j,\yb^i}=\zeta^i_j\}$. The computation of a suitable basis representation $L^i=\rge(Z_i^*,Y_i^*)$ can now be conducted analogously to Subsection \ref{SubsecPolyhedral}. It follows that we can employ the SCD-\ssstar Newton method from \cite[Section 7]{GfrOutVal23} as a part of our oracle. During the course of solving the inclusion $0\in F(x,y)+Q(y)$ in variable $y$, we  compute subspaces $\ee Lk\in\Sp^*(\partial q)(\ee yk, \ee {y^*}k)$ together with their basis representations at  iterates $(\ee yk,\ee{y^*}k)$ in the same way as described above, and the last  of these subspaces can be used for the computation of a pseudosubgradient in the oracle.

We solved first an academic example fully described in \cite[p.~1480]{GfrOutVal23}. To solve the problems, we used the code BTNCLC from the family of BT codes; see \cite{SZUG} and \cite{SZ}. With $n=3$, $l=4$ and $m=12$ and the required accuracy $\epsilon=10^{-6}$ we needed 44 iterations to obtain the productions displayed in Table~\ref{Tab1}.
\begin{table}[h]
% \begin{center}
\begin{tabular}{|c|c|c|c|c|c|}
\hline
&Leader&Follower 1&Follower 2&Follower 3&Follower 4\\
\hline
Commodity 1&61.7342&53.5983&20.4201&50.1653&44.7341\\
Commodity 2&78.6542&64.7525&29.8977&57.1949&49.8414\\
Commodity 3&47.8100&84.8591&49.6823&70.5791&59.9312\\
\hline
\end{tabular}

% \end{center}
	\caption{Optimal production for the Stackelberg strategy.\label{Tab1}}%
\end{table}

In Table \ref{Tab2} we present the values of the loss functions and compare them with the ones of the non-cooperative strategy computed in \cite{GfrOutVal23}. We see that the Leader only gains 17,98 units with the Stackelberg strategy, whereas the Followers loose between 31,25 units (Follower 2) and 92,87 units (Follower 1).
\begin{table}[ht]
% \begin{center}
\begin{tabular}{|c|c|c|c|c|c|}
\hline
Stackelberg strategy&Leader&Follower 1&Follower 2&Follower 3&Follower 4\\
&-2210.94& -2818.06 & -1978.79 & -2688.34& -2506.05\\
\hline
non-cooperative strategy&Player 1&Player 2&Player 3&Player 4&Player 5\\
&-2192.96 & -2910.93 & -2010.04 & -2767.10 & -2573.26\\
\hline
\end{tabular}
% \end{center}
\caption{Loss function values.\label{Tab2}}
\end{table}
This small problem can be solved within split seconds. In order to test the numerical efficiency of the suggested approach we increased the dimensions regardless the economic meaning of the corresponding model. We considered randomly generated problems for several combinations of players and commodities as in \cite{GfrOutVal23}, where as the Leader always the first firm was selected. We always have that $m=l\cdot n=1000$. As
starting point we chose $\ee x0=(1,,\ldots,1)^T\in\R^n$ and the stopping tolerance for the BT- algorithm was chosen as $\epsilon = 10^{-6}\norm{\ee \xi0}$, where $\ee\xi0$ denotes the pseudosubgradient returned by our oracle at $\ee x0$. The computations were performed in MATLAB on a single core (Intel i7-14700 CPU, 2.10 GHz) of a standard desktop.

\begin{table}[ht]
\begin{tabular}{|c||c|c|c|c|c|c|c|c|c|}
\hline
(l,n)&(5,200)&(10,100)&(20,50)&(25,40)&(40,25)&(50,20)&(100,10)&(200,5)\\
$\#$ oracles&216&127&115&100&94&67&55&119\\
CPU Time (s)&782.5&126.4&58.4&66.4&42.0&20.1&10.9&14.2\\
\hline
\end{tabular}
\caption{Numerical performance of BT for the Stackelberg game\label{Tab3}}
\end{table}

Note that due to the nonsmooth costs of change the Leader has a nonsmooth objective. Further, the evaluation of $S(x)$ amounts to the solution of a nonsmooth Nash equlibrium problem, which is a very difficult task itself. To the best of our knowledge there is no other method  able to solve an MPEC of such complexity and size.
\subsection{Shape optimization in three dimensional contact problems with Coulomb friction}
The aim of this example is to demonstrate the ability of the proposed method to solve MPECs with a nonsmooth objective and a large number of state variables. We have chosen a shape optimization problem from \cite{Has2}, where one considers a 3D elastic body being in possible contact with a rigid obstacle and subject to surface tractions on a part of its boundary. Our aim is to optimize the shape of another part of the boundary in such a way that, after loading, the $\ell_\infty$-norm of the stress along the respective contact area will be minimized. Whereas the infinite-dimensional model can be found, e.g., in \cite{Has2}, we work with the discrete (algebraic) model taken over from \cite{GfrMaOutVal23}. It attains the form of the GE
\begin{equation}\label{EqGECoulomb}
0\in A(x)y-l(x) + Q(x,y),
\end{equation}
where $x$ specifies the shape of the optimized boundary and $y$ comprises the displacements of all nodes of the considered body. $A$ is the stiffness matrix (dependent on $x$), $l$ models the given surface tractions and the multifunction $Q$ ensures
\begin{enumerate}
\item[(i)]the satisfaction of the unilateral (Signorini) contact condition and
\item[(ii)]the Coulomb friction law at nodes lying in the contact region.
\end{enumerate}

More precisely, we consider elastic bodies of the form
\[\Omega(x)=\{\xi\in\R^3\mv (\xi_1,\xi_2)\in[0,2]\times[0,1], \alpha(x;\xi_1,\xi_2)\leq \xi_3\leq 1\},\]
where $\alpha(x;\cdot,\cdot)$ describes the bottom surface of the body in dependence of our design variable $x$. In our test example, we take a uniform grid of control points $(\frac{k-1}4,\frac{j-1}4), k=1,\dots,9,\ j=1,\ldots 5\}$ on $[0,2]\times[0,1]$ and define
\[\alpha(x;\xi_1,\xi_2) := \sum_{k=1}^9\sum_{j=1}^5x_{kj}s_k(\xi_1)t_j(\xi_2),\ x\in\R^{9\times 5},\ (\xi_1,\xi_2)\in\R^2,\]
where $s_k$, $t_j$ are cubic splines specified by the conditions
\begin{align*}&s_k(\frac{i-1}4)=\delta_{ki}:=\begin{cases}0&\mbox{if $i\ne k$},\\1&\mbox{else,}\end{cases}\quad s_k''(0)=s_k''(2)=0,\ k,i=1,\ldots,9,\\
&t_j(\frac{i-1}4)=\delta_{ji},\ t_j''(0)=t_j''(1)=0,\ j,i=1,\ldots5.
\end{align*}
The rigid obstacle is given by the half-space $\{\xi\in\R^3\mv \xi_3\leq 0\}$ and the contact region of the body $\Omega(x)$ is its bottom surface $\{(\xi_1,\xi_2,\alpha(x;\xi_1,\xi_2))\mv(\xi_1,\xi_2)\in[0,2]\times[0,1]\}$.
The reference body $\Omega(0)$, a prism of size $2\times 1\times 1$, is uniformly carved into $n_1\times n_2\times n_3$ bricks yielding $(n_1+1)(n_2+1)(n_3+1)=:N$ nodes. The finite element discretization is performed by using trilinear elements and the partition of each $\Omega(x)$ was constructed from the
partition of the reference body $\Omega(0)$ by a suitable coordinate transformation in the $\xi_3$
direction ensuring that the mapping $(x,y)\mapsto A(x)y-l(x)$ is continuously differentiable. At the left boundary, i.e. for $\xi_1=0$, we assume Dirichlet boundary conditions (no displacements) and therefeore the number of state variables, i.e., the number of unknown displacements,  is $m=3 M$ with $M=n_1(n_2+1)(n_3+1)$. Further we have $p=n_1(n_2+1)$ nodes in the contact region without the Dirichlet part. We assume that the nodes of our discretization are numbered in such a way that the first $p$ nodes $\xi^1,\ldots,\xi^p$ are exactly the ones lying on the contact boundary and we use the following ordering for an arbitrary vector $y\in\R^m$:
\[y=(y^1,\ldots,y^p,y^R)\mbox{ with } y^i\in\R^3, i=1,\ldots,p,\ y^R\in\R^{3(M-p)}.\]
For every vector $z\in\R^3$ we denote by $z_{12}:=\left(\begin{smallmatrix}z_1\\z_2\end{smallmatrix}\right)$ the vector containing the first two elements. Then
\begin{gather*}Q(x,y) = \prod_{i=1}^p Q^i(x,y^i)\times\{0_{3(M-p)}\}\mbox{ with }
Q^i(x,y^i):=\left\{\begin{pmatrix}-{\cal F}\mu\partial\norm{y^i_{12}}
\\\mu\end{pmatrix}\bmv\mu\in N_{\R_+}\big(y^i_3+\alpha(x;\xi^i_1,\xi^i_2)\big)\right\}.\end{gather*}
Here ${\cal F}>0$ denotes the coefficient of Coulomb  friction. For each $i=1,\ldots,p$ let $\alpha^i$ denote the $9\times 5$ matrix given by $\alpha^i_{k,j}=s_k(\xi^i_1)t_j(\xi^i_2)$ so that $\alpha(x;\xi^i_1,\xi^i_2)=\sum_{k=1}^9\sum_{j=1}^5\alpha_{kj}^ix_{kj}$. In order to unburden the used notation we will identify $x$ and $\alpha^i$ for the moment as vectors from $\R^{45}$ resulting in $\alpha(x;\xi^i_1,\xi^i_2)=\skalp{\alpha^i,x}$.

Further we consider the mapping $\tilde Q:\R^3\tto\R^3$ defined by
\[\tilde Q(v)=\left\{\begin{pmatrix}-{\cal F}\mu\partial\norm{v_{12}}
\\\mu\end{pmatrix}\bmv\mu\in N_{\R_+}(v_3)\right\}.\]
Clearly,
\[\gph Q =\{\big((x,y),z\big)\bmv\left(\myvec{y^i_{12}\\\skalp{\alpha^i,x}+y^i_3}, z^i\right)\in\gph \tilde Q,\ i=1,\ldots,p,\ z^R=0\}\]
and we may conclude from \cite[Theorem 4.1]{GfrOut23} that $\Sp^*Q\big((x,y),z)$ is the collection of all subspaces expressed as $\rge(Z^*,X^*,y^*)$ with
\begin{align*}Z^*=\begin{pmatrix}{W^*}^1&&&0\\&\ddots&&\\&&{W^*}^p&\\0&&&I_{3(M-p)}\end{pmatrix},\
X^*=\begin{pmatrix}\alpha^1{V^*_3}^1&&&0\\&\cdots&&\\&&\alpha^p{V^*_3}^p&\\0&&&0\end{pmatrix},\
Y^*=\begin{pmatrix}{V^*}^1&&&0\\&\ddots&&\\&&{V^*}^p&\\0&&&0\end{pmatrix},
\end{align*}
where $\rge({W^*}^i,{V^*}^i)\in\Sp^*\tilde Q\left(\left(\begin{smallmatrix}y^i_{12}\\\skalp{\alpha^i,x}+y^i_3\end{smallmatrix}\right), z^i\right)$, $i=1,\ldots,p$, and ${V^*_3}^i$ denotes the third row of the $3\times 3$ matrix ${V^*}^i$ (row vector). For the computation of subspaces contained in $\Sp^*\tilde Q\left(\left(\begin{smallmatrix}y^i_{12}\\\skalp{\alpha^i,x}+y^i_3\end{smallmatrix}\right), z^i\right)$ and their basis representation we refer to \cite{GfrMaOutVal23}. Since $\tilde Q$ is an SCD mapping \cite[Proposition 6.6]{GfrMaOutVal23}, $Q$ is an SCD mapping as well. By using similar arguments as in \cite[Proposition 6.3]{GfrMaOutVal23} one can show that $\gph Q$ is a closed subanalytic set and therefore $Q$ is SCD \ssstar at every point of its graph. In particular, Assumption (A3) is fulfilled.

Considering the design variable $x$ again as a $9\times 5$ matrix, the set $U_{ad}$ is given by the relationships
\begin{align}
\label{EqXNonpenetration}  &x_{kj}\geq0,\ k=1,\ldots,9,\ j=1,\ldots,5,\\
\label{EqXMeanDrichlet} &\frac 15\sum_{j=1}^5x_{1j}=0.01,\\
\label{EqXMeanRest} &\frac 1{40}\sum_{k=2}^9\sum_{j=1}^5x_{kj}=0.01.
\end{align}
Condition \eqref{EqXNonpenetration} ensures that the non-deformed body $\Omega(x)$ does not penetrate the rigid obstacle at the control points. Equations \eqref{EqXMeanDrichlet} and \eqref{EqXMeanRest} require that the mean distance of the body $\Omega(x)$ to the obstacle at the control points belonging to the Dirichlet boundary and the remainder of the bottom surface, respectively, equals to $0.01$. Clearly, Assumption (A2) is fulfilled.

For a solution $y$ to the GE \eqref{EqGECoulomb} for given $x$, the normal stress at the boundary point $\xi^i$, $i=1,\ldots,p$, is given by $\big(A(x)y-l(x)\big)^i_3/w_i$, where $w_i$ is some positive integration weight depending on the mesh size of our discretization.
Thus
\[\varphi(x,y)=\max_{i=1,\ldots,p}\frac1{w_i}\big(A(x)y-l(x)\big)^i_3\]
and, since $\R^p\ni s\mapsto \max_{i=1,\ldots,p}s_i$ is semismooth with respect to Clarke's subdifferential and $A(x)y-l(x)$ is continuously differentiable, we conclude from Lemma \ref{LemChainRule} that Assumption (A5) is satisfied.

There remains  to mention that Assumptions (A1) and (A4) are fulfilled whenever $\cal F$ is sufficiently small, cf. \cite[Theorems 3.8, 3.13]{Has2}.

\begin{figure}
    \centering
    \begin{minipage}[c]{.39\textwidth}
\centering
\includegraphics[width=0.99\textwidth]{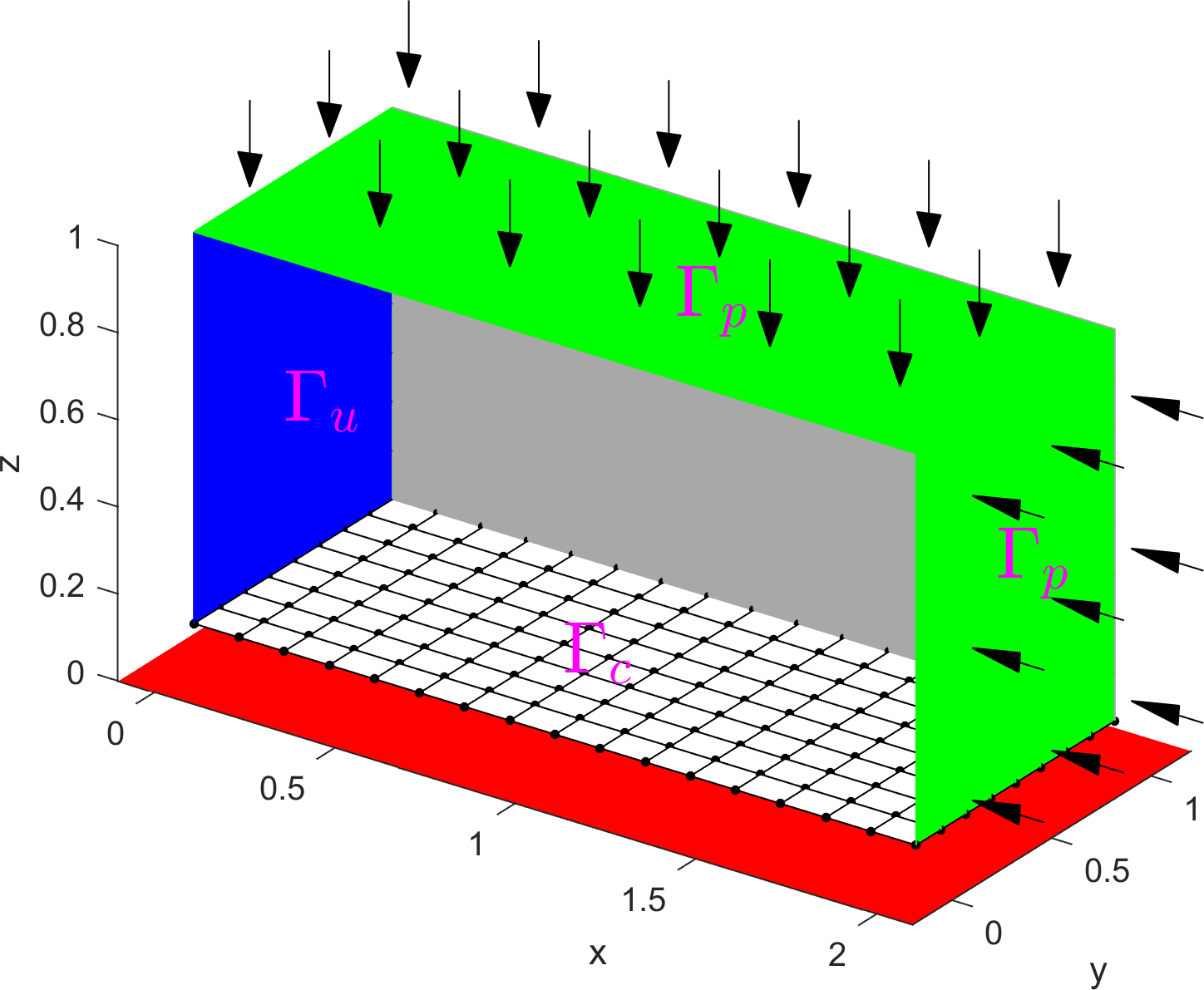}
\end{minipage}
\hspace{.01\textwidth}
\begin{minipage}[c]{.58\textwidth}
\vspace{0.5cm}
\centering
\includegraphics[width=0.99\textwidth]{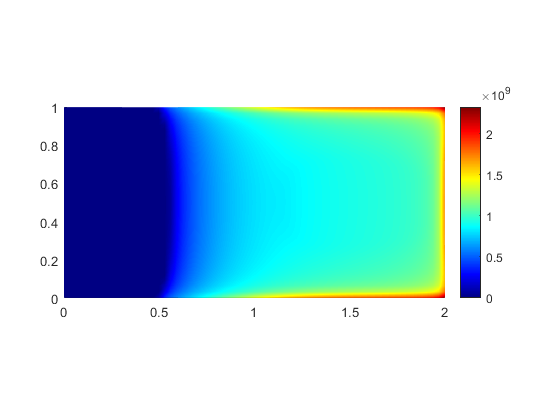}
\end{minipage}
\caption{\label{FigSetup}An undeformed cuboid domain is shown in the left pictures. The zero
Dirichlet boundary condition for displacements is assumed on the blue part of the boundary $\Gamma_u$, surface
tractions are applied to the green part of the boundary $\Gamma_p$ and the contact boundary $\Gamma_c$ is pressed against the (red) rigid plane foundation. Front faces are not visualized.\\
The distribution of normal stress in the plane bottom surface is shown in the right picture.}
\end{figure}
In our numerical test example we chose surface traction acting at the top boundary $\xi_3=1$ and at the right boundary $\xi_1=2$ pressing the body against the obstacle and the Dirichlet boundary as illustrated in Figure \ref{FigSetup}. The distribution of the normal stress for the  bottom surface given by $x_{kj}=0.01$, $k=1\ldots,9$, $j=1,\ldots,5$, i.e., a plane with distance $0.01$ to the obstacle, is also depicted there. For this plane bottom surface, the maximal normal stress of $2.439\cdot10^9$ is attained at the two corners of the right edge.
In dependence of the discretization level $lev$, the discretization is determined by
\[n_1=\lceil 4\cdot 2^{lev/2}\rceil,\ n_2=n_3=\lceil 2\cdot 2^{lev/2}\rceil.\]
We computed the optimal design $x$ for all levels $lev =3,\ldots,8$, where for $lev=3$ we took as starting point the design $x_{kj}=0.01$, $k=1,\ldots,9$, $j=1,\ldots,5$. For levels $lev>3$ we chose as starting point the optimal solution $x$ from the previous level. In Table \ref{TabShapeOptim} we display the number of state variables $m$ and the iterations, the oracle calls and CPU time needed by the BT-algorithm BTNCLC \cite{SZUG} with stopping tolerance $\epsilon =10^{-6}\norm{\ee \xi 0}$, where $\ee\xi0$ is the subgradient computed at the starting point. Further, the optmal objective function value $\varphi_{\rm opt}$ is given.
\begin{table}[h]
    \centering
    \begin{tabular}{|l|c|c|c|c|c|c|c|}
    \hline
        $lev$ & 3 & 4 & 5 & 6 & 7 & 8 \\
        $m$ & 1\,764 & 3\;888  & 11\,661 & 27\,744 & 79\,488 & 209\,088\\
        iterations& 134 & 97 & 104 & 400 & 643 & 853\\
        oracle calls& 139 & 97 & 104 & 414 & 675 & 902\\
        CPU time (s)& 57.7 &111.7 &458.0 & 5\,360 & 39\,340& 285\,760\\
        $\varphi_{\rm opt}$&$1.1806\cdot10^9$&$1.1251\cdot 10^9$&$1.0880\cdot10^9$& $1.0523\cdot10^9$&$1.0419\cdot10^9$&$1.0418\cdot 10^9$\\
        \hline
    \end{tabular}
    \caption{Performance of the BT-algorithm for various discretization levels.}
    \label{TabShapeOptim}
\end{table}
In the oracle calls the respective GE was solved with the  SCD \ssstar Newton method from \cite{GfrMaOutVal23}. We can see that for the finest discrtization level $lev=8$ one oracle call took more than 5 minutes. In Figure \ref{FigOptDesign} we present the optimized design of the bottom surface as well as the destribution of the normal stress on it. As expected, we now have a large part  of contact  where the maximal normal stress is attained (depicted in dark red).
\begin{figure}
       \centering
    \begin{minipage}[c]{.445\textwidth}
\centering
\includegraphics[width=0.99\textwidth]{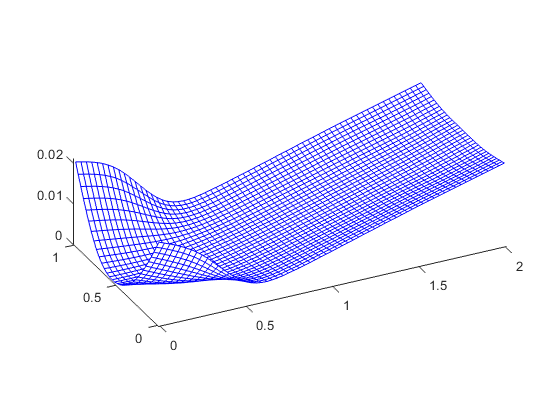}
\end{minipage}
\hspace{.01\textwidth}
\begin{minipage}[c]{.53\textwidth}
\vspace{0.5cm}
\centering
\includegraphics[width=0.99\textwidth]{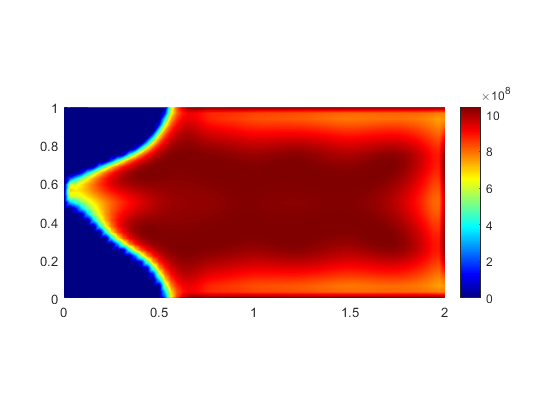}
\end{minipage}
\caption{\label{FigOptDesign}Optimal design (left) and distribution of normal stress in the optimized bottom surface (right).}

\end{figure}

\section{Bilevel programming}
In this section we are dealing with an optimistic bilevel program of the form
% \samepage{
\begin{align}\label{EqBiLev}\min\;&\varphi(x,y)\\
\nonumber\mbox{subject to }&y\in\argmin_y\{\psi(x,y)+q(x,y)\},\\
\nonumber&x\in U_{ad},
\end{align}
% }%
where $\varphi:\R^n\times\R^m\to\R$ is locally Lipschitz, $\psi:\R^n\times\R^m\to\R$ is continuously differentiable and $q:\R^n\times\R^m\to\oR$ is an lsc function. Consider now the following assumption.

\begin{myitems}
  \item[\rm(B1)]There is an open set $\Omega\supset U_{ad}$ such that for every $x\in\Omega$ the set of optimal solutions
  \[\argmin_y\{\psi(x,y)+q(x,y)\}=:\{\sigma(x)\}\]
  is a singleton. Moreover, for every $\xb\in\Omega$ the function $q(\cdot,\sigma(\xb))$ is continuous at $\xb$ and there is a neighborhood $U$ of $\xb$ such that the set of global solutions $\{\sigma(x)\mv x\in U\}$ is bounded.
\end{myitems}
\begin{lemma}\label{LemContSigma}
  Under Assumption {\rm(B1)} the mapping $\sigma$ is continuous on $\Omega$.
\end{lemma}
\begin{proof}
  Consider $\xb\in \Omega$. By the posed assumption, we can find a neighborhood $U$ and a compact set $K\subset\R^m$ such that $\{\sigma(x)\mv x\in U\}\subset\inn K$. Then we may add to $q(x,y)$ the indicator function $\delta_K(y)$  without changing the optimal solution sets for $x\in U$ and
  for the level set of the lower-level problem one has
  \[\{y\mv \psi(x,y)+q(x,y)+\delta_K(y)\leq \alpha\}\subset K,\ x\in U,\ \alpha\in\R.\]
  Now continuity of $\sigma$ at $\xb$ follows from \cite[Theorem 1.17]
  {RoWe98}.
\end{proof}
Obviously, $\sigma(x)$ satisfies the first-order optimality conditions
\[0\in\nabla_y\psi(x,\sigma(x))^T+\partial_y\, q(x,\sigma(x))\]
and is therefore a solution of the GE \eqref{eq-2} with $f:=\nabla_y\psi^T$ and $Q:=\partial_y\, q$. However, the solution set $S(x)$ consists of all stationary solutions of the lower-level program and might not be a singleton. Nevertheless, if all the assumptions (A2)--(A5) are fulfilled with $S(x)$ replaced by $\sigma(x)$, Theorem \ref{ThConvBT} remains valid by virtue of Corollary \ref{CorSSImplMap}.

\begin{example}
  Consider the bilevel program \eqref{EqBiLev} with $n=1$, $m=2$ and
  \begin{gather*}\varphi(x,y)= -x+y_1,\ U_{ad}=[-1,1],\\
  \psi(x,y)=\frac 12 y_1^2-\frac 12 y_2^2-xy_1,\quad q(x,y)=\delta_{C}(y)\mbox{ with }C:=\{(y_1,y_2)\mv -\frac 14 y_1\leq y_2\leq \frac 12y_1\}.
  \end{gather*}
  Then
  \[\argmin_y\{\psi(x,y)+q(x,y)\}=\{\sigma(x)\}=\begin{cases}\{(\frac 43 x,\frac 23 x)^T\}&\mbox{if $x\geq0$,}\\
  \{(0,0)^T\}&\mbox{if $x<0$}\end{cases}\]
  and it follows that Assumption (B1) is fulfilled. For the reduced objective function of the leader we obtain
  \[\vartheta(x)=\begin{cases}\frac 13 x&\mbox{if $x\geq 0$}\\
  -x&\mbox{if $x<0$}\end{cases}\]
  and hence $(\xb,\yb)=(0,(0,0))$ is a solution of the problem.

  However, the mapping $S(x)=\{y\mv 0\in\nabla_y\psi(x,y)^T+\partial_y\, q(x,y)\}$ of stationary solutions of the lower-level problem can be computed as
  \[S(x)=\begin{cases}\{\sigma(x),(\frac{16}{15}x,-\frac 4{15}x)^T, (x,0)^T\}&\mbox{if $x\geq 0$,}\\ \{\sigma(x)\}&\mbox{if $x<0$,}\end{cases}\]
  showing that $S$ is not single-valued. Further, there does not even exist a single-valued localization of $S$ around $(0,(0,0))$. Nevertheless, Assumptions (A2)-(A5) are fulfilled with $S$ replaced by $\sigma$, $f(x,y)=\nabla_y\psi(x,y)^T$ and $Q(x,y)=\partial_y q(x,y)=N_C(y)$. Clearly, (A2) and (A5) are fulfilled. Since $\delta_{C}$ is  convex lsc and $\gph Q=\R^n\times\gph N_C$ is the union of finitely many convex sets, (A3) holds by virtue of Example \ref{ExSCD} and Proposition \ref{PropSSstar}. Since $C$ is a convex polyhedral set, the adjoint SC derivative of $Q=N_C$ can be computed via
  \begin{align*}&\Sp^* N_C(\sigma(x),-\nabla_y \psi(x,\sigma(x))^T)
  =\begin{cases}\{L_1^*\}&\mbox{if $x>0$,}\\
  \{L_2^*\}&\mbox{if $x<0$,}\\
  \{L_1^*,L_2^*,L_3^*,L_4^*\}&\mbox{if $x=0$}
  \end{cases}
  \end{align*}
  with $L_2^*=\rge(0,I_2)$, $L_4^*=\rge(I_2,0)$ and
  \[L_1^*=\rge\left(\begin{pmatrix} 1&0\\1/2&0\end{pmatrix},\begin{pmatrix}0&1/2\\ 0&-1\end{pmatrix}\right),\
  L_3^*=\rge\left(\begin{pmatrix}1&0\\-1/4&0\end{pmatrix},\begin{pmatrix}0&1/4\\ 0&1\end{pmatrix}\right);  \]
  cf. Section \ref{SubsecPolyhedral}.
  The matrices $\nabla_yf(x,\sigma(x))^TZ^*+Y^*$ which have to be checked in (A4) are therefore given by
  \[
    \begin{pmatrix}1&1/2\\-1/2&-1\end{pmatrix},\ I_2,\ \begin{pmatrix}1&1/4\\1/4&1\end{pmatrix}, \begin{pmatrix}1&0\\0&-1\end{pmatrix}
    \]
    and are obviously nonsingular. Thus (A4) also holds and it follows that we can solve this bilevel optimization problem with the BT-algorithm using the ImP approach, provided we can actually compute a global solution of the lower-level problem, which is of course a very stringent assumption.
\end{example}

\begin{example}\label{ExProj}
As a numerical experiment let us consider a bilevel program of the type \eqref{EqBiLev}, where the lower level objective is strictly convex in $y$ and therefore the map $S(x)$ is single-valued. We set $n=m=3$ and
\[\psi(x,y):=\frac 12(y-y_0)^TC(x)(y-y_0),\ \  q(x,y):=\delta_{\Gamma(y)},\]
where $y_0\in\R^3$ is a given "target" point. The $3\times 3$ diagonal matrix $C(x)={\rm diag}(x_1,x_2,x_3)$ depends on $x$ and $\Gamma$ is a closed convex set given by the five inequalities
\[
    \frac12y_1^2 \pm y_1 - y_3\leq 0,\quad \frac12 y_2^2 \pm y_2 - y_3\leq 0,\quad y_3\geq 0.
\]
Note that LICQ is violated at $y=(0,0,0)^T$, but both MFCQ and CRCQ are fulfilled. This example demonstrates the application of the results from Subsection \ref{SubsecInequ}.
For our numerical experiments, we chose a nondifferentiable function $\varphi(x,y):=\norm{y}_1$  and $U_{ad}:=[1,50]^3$.

We used the ImP-approach as described in Section \ref{SecImP}. The subspaces $L^*\in \Sp^*Q((x,S(x)),-f(x,S(x)))$ were computed as explained in Section \ref{SubsecInequ}. The resulting nonsmooth optimization problem was solved by the code BTNCBC \cite{SZUG}. For the lower-level quadratic optimization problems we used the software Gurobi \cite{gurobi}.

 \begin{figure}[h]
     \begin{center}
         {\includegraphics[width=0.45\hsize]{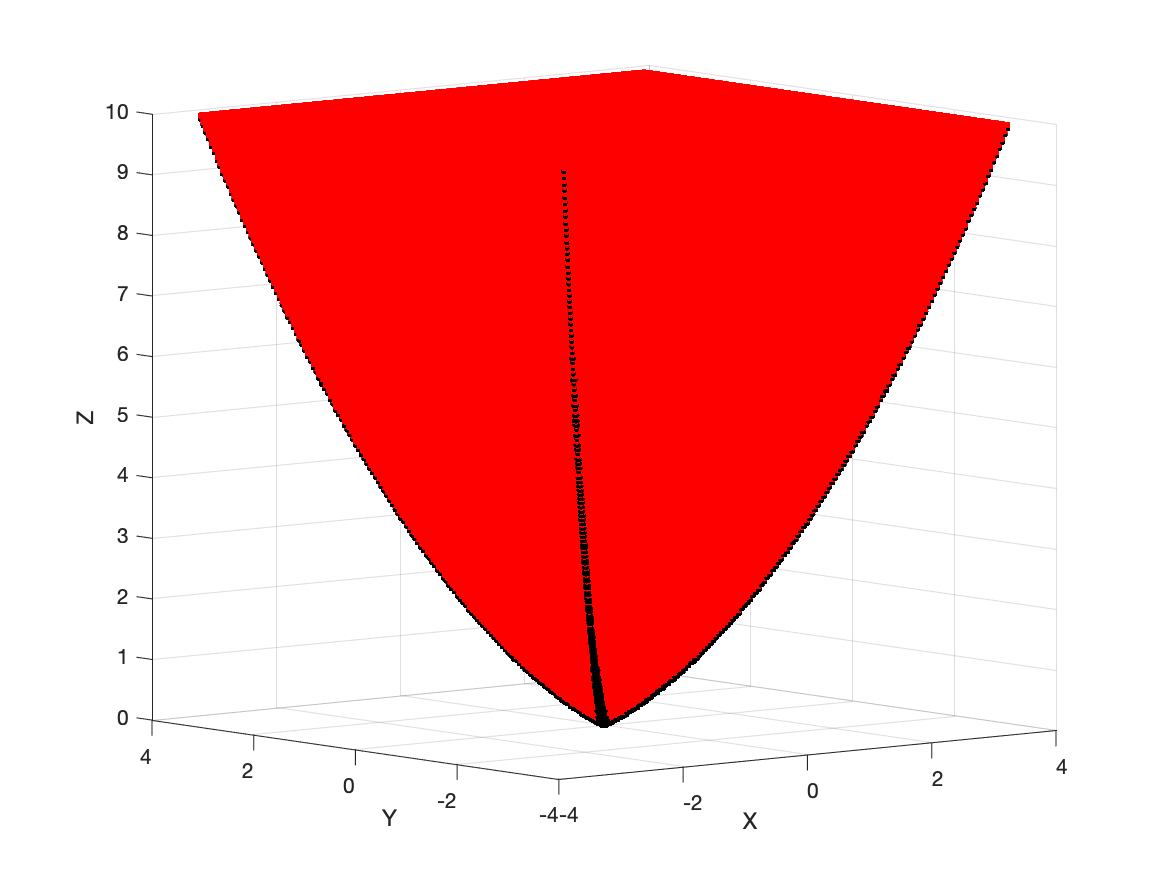}}
         {\includegraphics[width=0.45\hsize]{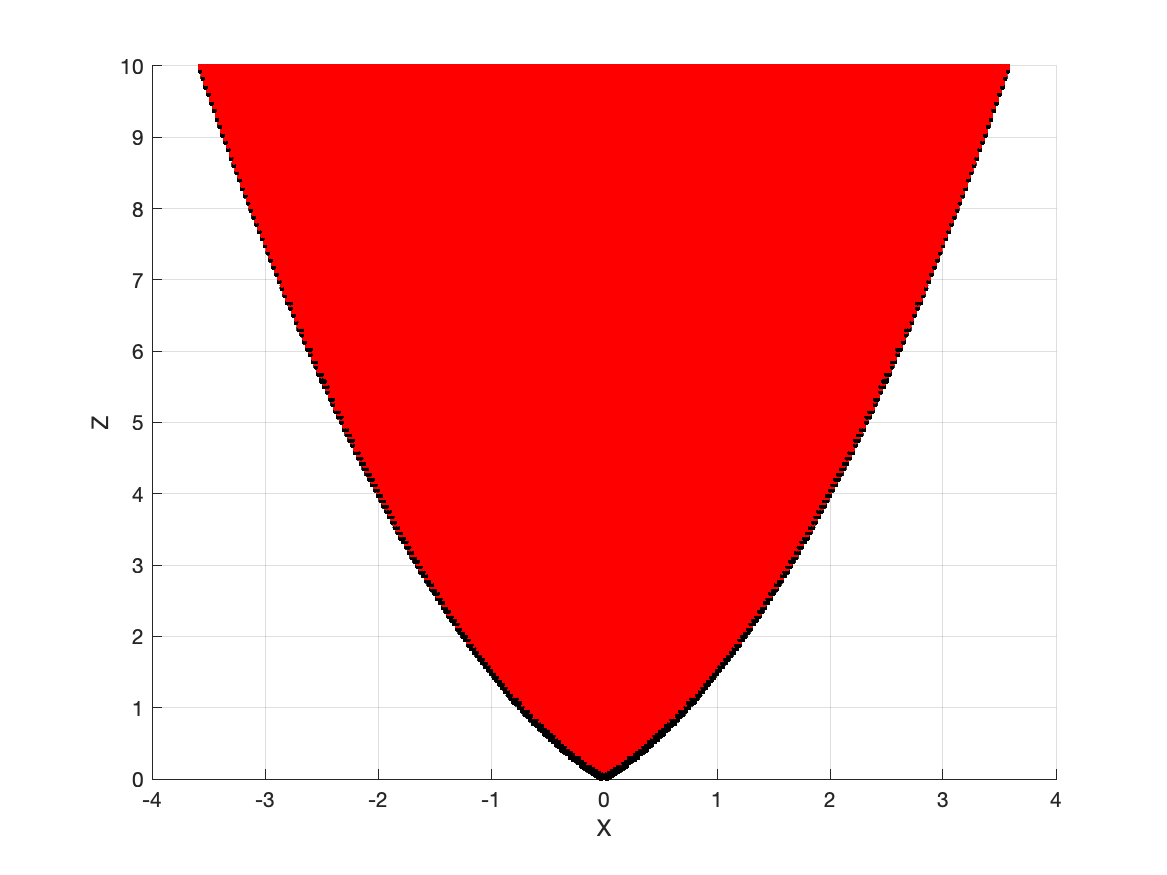}}
     \end{center}
     \caption{Set $\Gamma$ from Example~\ref{ExProj} in two different views.}
     \label{fig:proj}
 \end{figure}
We solved the problem with two target  points, $y_0^{(1)} = (1,\;2,\;3)^T$ and
$y_0^{(2)} = (1,\;2,\;-0.5)^T$.  In both cases, the initial point for BTNCBC was $x=(3,\;3,\;3)^T$.

In case of $y_0^{(1)}$, BTNCBC needed 8 iterations to reach the required accuracy $10^{-6}$. As expected, at the optimum the third component of $x$ was on the upper bound of $U_{ad}$: $\hat x=(3.000013,\, 3.000013,\, 49.999994)$ with corresponding $\hat y=
   (1.000004,\,
   1.648760,\,
   3.007966)^T$.

For $y_0^{(2)}$, BTNCBC reached the solution in 18 iterations giving $\hat x=(         3.000001,\,
   3.000001,\,
  31.270435)$ and  $\hat y$ with components of order $10^{-11}$.
\end{example}

\begin{remark}
For the numerical solution of the above problem one could alternatively use the approach from \cite[Section 6.2]{Dem} provided the upper objective $\varphi$ were continuously differentiable. Note that in our approach  the adjoint equation (\ref{EqOracle1}) has a smaller dimension than its counterpart in \cite{Dem}.
\end{remark}

\if{
\section{On a semismooth Newton method for decomposable optimization problems}
Consider the problem
\begin{align}
  \min_{x,y}\varphi(x,y)
\end{align}
where $\varphi:\R^n\times\R^m\to \oR$ is an lsc function. Assuming that minimization with respect to $y$ can be efficiently conducted, one could attempt to eliminate $y$ and to solve the reduced problem
\begin{equation}\label{EqRedProbl}\min_x \vartheta(x)\quad\mbox{with}\quad \vartheta(x):=\min_y \varphi(x,y).\end{equation}
 Under reasonable assumptions, the function $\vartheta$ is continuously differentiable and the gradient $\nabla\vartheta$ can be easily computed. Thus, numerous first-order methods are available for minimizing $\vartheta$ and one might ask whether there are second-order  methods yielding faster convergence. A candidate for such a second-order method would be the semismooth Newton approach. If $\nabla\vartheta^T$ is semismooth with respect to a mapping $\Theta:\R^n\tto\R^{n\times n}$, the semismooth Newton iteration is given by
\begin{equation}\label{EqSSNewt}\ee x{k+1}=\ee xk-{\ee Gk}^{-1}\nabla \vartheta(\ee xk)^T\quad \mbox{with}\quad \ee Gk\in\Theta(\ee xk).\end{equation}
If at a solution $\hat x$ of \eqref{EqRedProbl} all matrices $G\in(\cl\Theta)(\hat x)$ are nonsingular, there exists a neighborhood $U$ of $\hat x$ such that for all starting points $\ee x0\in U$ the sequence $\ee xk$ given by \eqref{EqSSNewt} is well defined and converges superlinearly to $\hat x$; cf.\ \cite{Ulb11}.

A crucial part of this semismooth Newton approach is the computation of the mapping $\Theta$ and we will explain in the sequel how this can be done.
Consider the following assumptions:
\begin{myitems}
  \item[\rm(SS1)] $\varphi(x,y)=\psi(x,y)+q(y)$, where $q:\R^m\to\oR$ is lsc and $\psi:\R^n\times\R^m\to\R$ is twice continuously differentiable.
  \item [\rm(SS2)] There is an open set $\Omega\subset\R^n$ such that
  \[\argmin_{y\in\R^m}\varphi(x,y)=\{\sigma(x)\}\quad\mbox{is a singleton for every $x\in \Omega$.}\]
  \item [\rm(SS3)] For every $\xb\in \Omega$ there is a neighborhood $U$ of $\xb$ such that the set $\{\sigma(x)\mv x\in U\}$ is bounded.
\end{myitems}
\begin{lemma}
  Assume that {\rm(SS1)-(SS3)} are fulfilled. Then the mapping $\sigma$ is continuous on $\Omega$ and the function $\vartheta$ given by \eqref{EqRedProbl} is continuously differentiable on $\Omega$ with  $\nabla\vartheta(x)=\nabla_x \psi(x,\sigma(x))$.
\end{lemma}
\begin{proof}
  The continuity of $\sigma$  follows  from Lemma \ref{LemContSigma}. Further, $\partial^\infty\varphi(x,\sigma(x))=\{0\}\times  \partial^\infty q(\sigma(x))$ and hence there does not exist $x^*\not=0$ with $(x^*,0)\in\partial^\infty\varphi(x,S(x))$, where for the definition and properties of the horizon subdifferental $\partial^\infty$ we refer the reader to \cite{RoWe98}.  We may now  conclude from \cite[Corollary 10.14]{RoWe98} that $\vartheta$ is strictly differentiable at every $x\in\Omega$ with $\nabla \vartheta(x)=\nabla_x \psi(x,\sigma(x))$. Together with the continuity of $\sigma$ we conclude that $\vartheta$ is smooth on $\Omega$.
\end{proof}

In order to apply Corollary~\ref{CorSSImplMap} with $f=\nabla_y\psi^T$ and $Q(x,y)=\partial q(y)$, we need two additional assumptions:
\begin{myitems}
  \item[\rm(SS4)] For every $x\in \Omega$, the subgradient mapping $\partial q$ is SCD-\ssstar at $(\sigma(x),-\nabla_y \psi(x,\sigma(x)))$.
  \item[\rm(SS5)]  For every $x\in\Omega$ and for every subspace $L^*=\rge(Z^*,Y^*)\in\Sp^*(\partial q)(\sigma(x),-\partial_y\psi(x,\sigma(x))^T)$, where $Z^*,Y^*\in \R^{m\times m}$, there holds that the $m\times m$ matrix $\nabla^2_{yy} \psi(x,\sigma(x))Z^*+Y^*$ is nonsingular.
\end{myitems}
\begin{theorem}
      Under Assumptions {\rm(SS1)--(SS5)}, the mapping $\sigma$ is semismooth on $\Omega$ with respect to the mapping $\Psi:\R^n\tto\R^{m\times n}$ given by
  \[\Psi(x):=\begin{cases}\begin{array}{l}\{-(\nabla^2_{yy} \psi(x,y)Z^*+Y^*)^{-T}({Z^*}^T\nabla^2_{yx} \psi(x,y))\mv\\
   \qquad\qquad\rge(Z^*,Y^*)\in \Sp^*(\partial q)(\sigma(x),-\nabla_y\psi(x,\sigma(x))^T)\}\end{array}&\mbox{if $x\in \Omega$},\\\emptyset&\mbox{else.}\end{cases}\]
  Further, $\nabla \vartheta(x)^T=\nabla_x\psi(x,\sigma(x))^T$ is  semismooth on $\Omega$ with respect to the mapping $\Theta:\R^n\tto\R^{n\times n}$ given by
  \begin{equation*}%\label{EqTheta}
  \Theta(x):=\begin{cases}
    \{\nabla^2_{xx}\psi(x,\sigma(x)) +\nabla^2_{xy}\psi(x,\sigma(x)) A\mv  A\in\Psi(x)\}&\mbox{if $x\in\Omega$},\\
    \emptyset&\mbox{else.}
  \end{cases}
  \end{equation*}
\end{theorem}
The proof follows from Corollary \ref{CorSSImplMap}. We may thus conclude that problem \eqref{EqRedProbl} can be solved numerically by the semismooth Newton approach given by \eqref{EqSSNewt}.
}\fi

\section{Conclusion}

The main contributions of the paper can be summarized as follows.
\begin{enumerate}
\item[(i)] We have substantially extended the class of MPECs, in which  pseudosubgradients of the reduced objective can be computed. This concerns both the lower-level problem and the upper-level objectives. The respective requirements are expressed in terms of problem data. Further, we have clarified the nature of stationarity conditions that are fulfilled by the accumulation points of sequences generated by the applied bundle algorithms.
\item[(ii)] Apart from MPECs, satisfying the standing assumptions, we have employed the achieved theoretical results also in a class of bilevel programming problems. This leads to reduced problems solvable by a bundle method as well. \if{Moreover, we have specified a class of decomposable optimization problems which, after the reduction, can be solved by a nonsmooth Newton-type method.}\fi
\item[(iii)] Finally, the consequent usage of the recently developed theory of SCD mappings has enabled us to weaken the imposed assumptions and to simplify the solution of adjoint equations for the computation of pseudosubgradients.
\end{enumerate}

\if{
Let us finish with a few observations.

The paper contains a certain ``minimal'' set of assumptions, under which a class of MPECs and bilevel programs can be numerically solved by the ImP approach combined with a bundle method.

In some MPECs the lower-level problem may be solved by the SCD-\ssstar Newton method (\cite{GfrOut22a}). In such a case, in each iteration of the bundle method a suitable subspace $L^*$ is available and the computation of pseudosubgradients reduces to the solution of the linear square system (\ref{EqOracle1}).

Nonsmooth upper-level objectives may appear, e.g., in connection with possible state constraints, augmented to the objective via an exact penalty; see \cite{CO}.
}\fi

\section*{Declarations}

\noindent{\bf Conflict of interest.} The authors report that there are no competing interests to declare.

\noindent{\bf Data availability. } The data used in Section \ref{SecStackel} are available from \cite{GfrOutVal23}.

\noindent{\bf Funding.}
The work of M.~Ko\v{c}vara and J.~V.~ Outrata has been supported by the Czech Science Foundation through project No.\ 22-15524S.


\begin{thebibliography}{99}
\bibitem{BM}
{\sc M. Benko, P. Mehlitz}, {\em On implicit variables in optimization theory}, J. of Nonsmooth Analysis and Optimization 2 (2021), 7215.
%
\bibitem{Has2}
{\sc P. Beremlijski, J. Haslinger, M. Ko\v{c}vara, R. Ku\v{c}era, J.~V. Outrata},
{\em Shape optimization in 3D contact problems with Coulomb friction}, SIAM J. Optimization 20(2009), pp.~416--444.
%%
\bibitem{ChNaQui00}
{\sc X. Chen, Z. Nashed, L. Qi}, {\em Smoothing methods and semismooth methods for
nondifferentiable operator equations}, SIAM J. Numer. Anal. 38 (2000), pp.~1200--1216.
%
\bibitem{Cla83}
{\sc F.~H. Clarke}, {\em Optimization and nonsmooth analysis}, John Wiley \& Sons, New York, Chichester, Brisbane, Toronto, Singapore, 1983.
%
\bibitem{Dem}
{\sc S. Dempe}, {\em Foundations of Bilevel Programming}, Kluwer, New York, 2002.
%
\bibitem{DeBa01}
{\sc S. Dempe, J.~F. Bard}, {\em Bundle trust-region algorithm for
bilinear bilevel programming}, J. Optim. Theory Appl., 110 (2001), pp.~265--288.
%
\bibitem{DR}
{\sc  A.~L. Dontchev, R.~T. Rockafellar}, {\em Implicit Functions and Solution Mappings}, Springer, Heidelberg, 2014.
%
\bibitem{GfrOut21}
{\sc H. Gfrerer, J.~V. Outrata}, {\em On a semismooth$^*$ Newton method for solving generalized equations}, SIAM J. Optim. 31 (2021), pp.~489--517.
%
\bibitem{Gfr25a}
{\sc H. Gfrerer}, {\em On a globally convergent semismooth$^*$ Newton method in nonsmooth nonconvex optimization}, Comput. Optim. Appl. 91 (2025), pp.~67--124.
%
%
\bibitem{GfrMaOutVal23}
{\sc H. Gfrerer, M. Mandlmayr, J.~V. Outrata, J. Valdman}, {\em On the SCD semismooth$^*$ Newton method for generalized equations with application to a class of static contact problems with Coulomb friction}, Comput. Optim. Appl. 86 (2023), pp.~1159--1191.
%
%
\bibitem{GfrOut16}
{\sc H. Gfrerer, J.~V. Outrata}, {\em On computation of generalized derivatives of the normal-cone mapping and their applications}, Mathematics of Operations Research 41 (2016), pp. 1535--1556.
%
\bibitem{GfrOut22a}
{\sc H. Gfrerer, J.~V. Outrata}, {\em On (local) analysis of multifunctions via subspaces contained
in graphs of generalized derivatives}, J. Math. Anal. Appl. 508 (2022), article no. 125895.
%
\bibitem{GfrOut23}
{\sc H. Gfrerer, J.~V. Outrata}, {\em On the isolated calmness property of implicitly defined multifunctions},  J. Convex Anal. 30 (2023), pp.~1001--1023.
%
\bibitem{GfrOut26a}
{\sc H. Gfrerer, J.~V. Outrata}, {\em  On the role of semismoothness in nonsmooth numerical analysis: Theory},  arXiv:2405.14637 (2026), to appear in SIAM J. Optim.
%
\bibitem{GfrOutVal23}
{\sc H. Gfrerer, J.~V. Outrata, J. Valdman}, {\em On the application of the SCD semismooth$^*$ Newton method to variational inequalities of the second kind}, Set-Valued Var. Anal. 30 (2022), pp.~1453--1484.
%
%
\bibitem{Gow04}
{\sc M.~S. Gowda}, {\em  Inverse and implicit function theorems for H-differentiable and semismooth functions},
Optim. Methods Softw. 19  (2004), pp.~443--461.
%
\bibitem{gurobi}
{\sc {Gurobi Optimization, LLC}}, {\em Gurobi Optimizer Reference Manual}, (2024), \url{https://www.gurobi.com}.
\bibitem{Hen1}
{\sc R. Henrion, A.~Y. Kruger, J.~V. Outrata}, {\em Some remarks on stability of generalized equations},  J. Optim. Theory Appl. 159 (2013), pp.~681--697.
%
\bibitem{HiItKu03}
{\sc M. Hinterm\"uller, K. Ito, K. Kunisch}, {\em The primal-dual active set method as a
semismooth Newton method}, SIAM J. Optimization 13 (2003), pp.~865--888.
%
\bibitem{KO}{\sc M. Kočvara, J.~V. Outrata}, {\em  Inverse truss design as a conic mathematical program with equilibrium constraints},  Discrete and Continuous Dynamical systems - Series S, 10(2017), pp.~1329--1350.
%
\bibitem{Kum00}
{\sc B. Kummer}, {\em Generalized Newton and NCP-methods: Convergence, regularity, actions},
Discussiones Mathematicae - Differential Inclusions
20 (2000), pp.~209–244.
%
%
\bibitem{LuPaRa97}
 {\sc Z.-Q. Luo, J.-S. Pang, D. Ralph}, {\em Mathematical Programs with Equilibrium
Constraints,} Cambridge University Press, Cambridge,  1997.
%
\bibitem{MiGuNo87}
{\sc V.~S. Mikhalevich, A.~M. Gupal, V.~I. Norkin}, {\em Methods of Nonconvex Optimization}, Nauka, Moscow, Russia, 1987 (in Russian).
%
%
\bibitem{Mo06a}
{\sc B.~S. Mordukhovich}, {\em Variational Analysis and Generalized Differentiation I: Basic Theory}, Springer, Berlin, 2006.
%
\bibitem{NJH}
{\sc J. Necas, J. Jarusek, J. Haslinger},
{\em On the solution of the variational inequality to the Signorini problem with small friction}, Boll. Unione Mat. Ital. V. Ser.B, vol. 17 (1980), pp.~796--811.
%
\bibitem{No78}
{\sc V.~I. Norkin}, {\em Nonlocal minimization algorithms for nondifferentiable functions}, Kibernetika, No. 5 (1978), pp.~57--60.
%
\bibitem{No80}
{\sc V.~I. Norkin}, {\em Generalized-differentiable functions}, Kibernetika, No. 1(1980), pp.~9--11.
%

\bibitem{CO}
{\sc J.~V. Outrata, M. \v{C}ervinka}, {\em On the implicit programming approach in a class of mathematical programs with equilibrium constraints}, Control and Cybernetics 38 (2009), pp.~1557--1574.
%
\bibitem{OKZ}{\sc J.~V. Outrata, M. Ko\v{c}vara, J. Zowe}, {\em Nonsmooth Approach to Optimization Problems with Equilibrium Constraints}, Kluwer, Dordrecht, 1998.
%
\bibitem{RoWe98}
{\sc R.~T. Rockafellar, R.~J.-B. Wets}, {\em Variational Analysis}, Springer, Berlin, 1998.
%
\bibitem{SZUG}{\sc H. Schramm, J. Zowe}, {\em Bundle trust methods: Fortran codes for nondifferentiable optimization. User's guide}, Report 269, Mathematisches Institut, Universitaet Bayreuth.
%
\bibitem{SZ}{\sc H. Schramm, J. Zowe}, {\em A version of the bundle idea for minimizing a nonsmooth function: conceptual idea, convergence analysis, numerical results}, SIAM J. Optim. 2(1992), pp.121--152.
%
%
\bibitem{Ulb11}
{\sc M. Ulbrich}, {\em Semismooth Newton methods for variational inequalities and constrained optimization problems in function spaces}, MOS-SIAM Series on Optimization,  SIAM,  Philadelphia, 2011.
%
\bibitem {Zowe}
{\sc J. Zowe}, {\em The BT algorithm for minimizing a nonsmooth functional subject to linear constraints}, in  F. H. Clarke, V. F. Dem’yanov, F. Giannessi (eds.),  Nonsmooth Optimization and Related Topics, Springer, New York (1989), pp.459--480.

\end{thebibliography}
\end{document}